\documentclass{article}%
\usepackage{amssymb}
\usepackage{amsfonts}
\usepackage{amsmath}
\usepackage{graphicx}%
\newtheorem{theorem}{Theorem}

\newtheorem{definition}[theorem]{Definition}
\newtheorem{example}[theorem]{Example}

\newtheorem{lemma}[theorem]{Lemma}

\newtheorem{proposition}[theorem]{Proposition}
\newtheorem{remark}[theorem]{Remark}

\newenvironment{proof}[1][Proof]{\noindent\textbf{#1.} }{\ \rule{0.5em}{0.5em}}
\begin{document}

\title{ $ADE$ bundles over surfaces with \\$ADE$ singularities}
\author{Yunxia Chen \& Naichung Conan Leung}
\date{}
\maketitle

\begin{abstract}
Given a compact complex surface $X$ with an $ADE$ singularity and
$\mathit{p}_{g}=0$, we construct $ADE$ bundles over $X$ and its minimal
resolution $Y$. Furthermore, we descibe their minuscule representation bundles
in terms of configurations of (reducible) $(-1)$-curves in $Y$.

\end{abstract}

\section{Introduction}

It has long been known that there are deep connections between Lie theory and
the geometry of surfaces. A famous example is an amazing connection between
Lie groups of type $E_{n}$ and del Pezzo surfaces $X$ of degree $9-n$ for
$1\leq n\leq8$. The root lattice of $E_{n}$ can be identified with
$K_{X}^{\perp}$, the orthogonal complement to $K_{X}$ in $Pic(X)$.
Furthermore, all the lines in $X$ form a representation of $E_{n}$. Using the
configuration of these lines, we can construct an $E_{n}$ Lie algebra bundle
over $X$ \cite{LZ}. If we restrict it to the anti-canonical curve in $X$,
which is an elliptic curve $\Sigma$, then we obtain an isomorphism between the
moduli space of degree $9-n$ del Pezzo surfaces which contain $\Sigma$ and the
moduli space of $E_{n}$-bundles over $\Sigma$. This work is motivated from
string/$F$-theory duality, and it has been studied extensively by
Friedman-Morgan-Witten \cite{FM}\cite{FMW}\cite{FMW2}, Donagi \cite{BDL}%
\cite{BDO}\cite{CDW}\cite{D}, Leung-Zhang \cite{LXZ}\cite{LZ}\cite{LZ2} and
others \cite{CM}\cite{Lee}\cite{L1}\cite{L2}.

In this paper, we study the relationships between simply-laced, or $ADE$, Lie
theory and rational singularities of surfaces. Suppose%
\[
\pi:Y\rightarrow X
\]
is the minimal resolution of a compact complex surface $X$ with a rational
singularity. Then the dual graph of the exceptional divisor $\sum_{i=1}%
^{n}C_{i}$ in $Y$ is an$\ ADE$ Dynkin diagram. From this we have an $ADE$ root
system $\Phi:=\{\alpha=\sum a_{i}[C_{i}]|\alpha^{2}=-2\}$ and we can construct
an $ADE$ Lie algebra bundle over $Y$:
\[
\mathcal{E}_{0}^{\mathfrak{g}}:=O_{Y}^{\oplus n}\oplus\bigoplus_{\alpha\in
\Phi}O_{Y}(\alpha)
\]
Even though this bundle can not descend to $X$, we show that it can be
deformed to one which can descend to $X$ provided that $p_{g}\left(  X\right)
=0$.

\begin{theorem}
$($Proposition \ref{Lie bracket}, Proposition \ref{holomorphic}, Theorem
\ref{thm1} and Lemma \ref{surj}$)$

Assume $Y$ is the minimal resolution of a surface $X$ with a rational
singularity at $p$ of type $\mathfrak{g}$ and $C=\Sigma_{i=1}^{n}C_{i}$ is the
exceptional divisor. If $p_{g}\left(  X\right)  =0$, then

$(i)$ given any $(\varphi_{C_{i}})_{i=1}^{n}\in\Omega^{0,1}(Y,\bigoplus
_{i=1}^{n}O(C_{i}))$ with $\overline{\partial}\varphi_{C_{i}}=0$ for every
$i$, it can be extended to $\varphi=(\varphi_{\alpha})_{\alpha\in\Phi^{+}}%
\in\Omega^{0,1}(Y,\bigoplus_{\alpha\in\Phi^{+}}O(\alpha))$ such that
$\overline{\partial}_{\varphi}:=\overline{\partial}+ad(\varphi)$ is a
holomorphic structure on $\mathcal{E}_{0}^{\mathfrak{g}}$. We denote this new
holomorphic bundle as $\mathcal{E}_{\varphi}^{\mathfrak{g}}$.

$(ii)$ Such a $\overline{\partial}_{\varphi}$ is compatible with the Lie
algebra structure.

$(iii)$ $\mathcal{E}_{\varphi}^{\mathfrak{g}}$ is trivial on $C_{i}$ if and
only if $[\varphi_{C_{i}}|_{C_{i}}]\neq0\in H^{1}(C_{i},O_{C_{i}}(C_{i}%
))\cong\mathbb{C}$.

$(iv)$ There exists $[\varphi_{C_{i}}]\in H^{1}(Y,O(C_{i}))$ such that
$[\varphi_{C_{i}}|_{C_{i}}]\neq0$.

$(v)$ Such a $\mathcal{E}_{\varphi}^{\mathfrak{g}}$ can descend to $X$ if and
only if $[\varphi_{C_{i}}|_{C_{i}}]\neq0$ for every $i$.
\end{theorem}

\begin{remark}
Infinitesimal deformations of holomorphic bundle structures on $\mathcal{E}%
_{0}^{\mathfrak{g}}$ are parametrized by $H^{1}(Y,End(\mathcal{E}%
_{0}^{\mathfrak{g}}))$, and those which also preserve the Lie algebra
structure are parametrized by $H^{1}(Y,ad(\mathcal{E}_{0}^{\mathfrak{g}%
}))=H^{1}(Y,\mathcal{E}_{0}^{\mathfrak{g}})$, since $\mathfrak{g}$ is
semi-simple. If $p_{g}\left(  X\right)  =q(X)=0$, e.g. rational surface, then
for any $\alpha\in\Phi^{-}$, $H^{1}(Y,O(\alpha))=0$. Hence $H^{1}%
(Y,\mathcal{E}_{0}^{\mathfrak{g}})=H^{1}(Y,\bigoplus_{\alpha\in\Phi^{+}%
}O(\alpha))$.
\end{remark}

This generalizes the work of Friedman-Morgan \cite{FM}, in which they
considered $E_{n}$ bundles over generalized del Pezzo surfaces. In this paper,
we will also describe the minuscule representation bundles of these Lie
algebra bundles in terms of $(-1)$-curves in $Y$.

Here is an outline of our results. We first study $(-1)$-curves in $Y$ which
are (possibly reducible) rational curves with self intersection $-1$. If there
exists a $(-1)$-curve $C_{0}$ in $X$ passing through $p$ with minuscule
multiplicity $C_{k}$ (Definition \ref{min}), then $(-1)$-curves $l$'s in $Y$
with $\pi(l)=C_{0}$ form the minuscule representation$\footnote{Here $V$ is
the lowest weight representation with lowest weight dual to $-C_{k}$, i.e. $V$
is dual to the highest weight representation with highest weight dual to
$C_{k}$.}$ $V$ of $\mathfrak{g}$ corresponding to $C_k$ (Proposition
\ref{representation}). When $V$ is the standard representation of
$\mathfrak{g}$, the configuration of these $(-1)$-curves determines a
symmetric tensor $f$ on $V$ such that $\mathfrak{g\ }$is the space of
infinitesimal symmetries of $(V,f)$. We consider the bundle%
\[
\mathfrak{L}_{0}^{(\mathfrak{g},V)}:=\bigoplus_{\substack{l:(-1)-curve\\\pi
(l)=C_{0}}}O_{Y}(l)
\]
over $Y$ constructed from these $(-1)$-curves $l$'s. This bundle can not
descend to $X$ as it is not trivial over each $C_i$.\footnote{Unless specify
otherwise, $C_{i}$ always refers to an irreducible component of $C$, i.e.
$i\neq0$.}

\begin{theorem}
$($Theorem \ref{thm2} and Theorem \ref{thm3}$)$

For the bundle $\mathfrak{L}_{0}^{(\mathfrak{g},V)}$ with the corresponding
minuscule representation $\rho:\mathfrak{g}\longrightarrow End(V)$,

$(i)$ there exists $\varphi=(\varphi_{\alpha})_{\alpha\in\Phi^{+}}\in
\Omega^{0,1}(Y,\bigoplus_{\alpha\in\Phi^{+}}O(\alpha))$ such that
$\overline{\partial}_{\mathfrak{\varphi}}:=\overline{\partial}_{0}%
+\rho(\varphi)$ is a holomorphic structure on $\mathfrak{L}_{0}^{(\mathfrak{g}%
,V)}$. We denote this new holomorphic bundle as $\mathfrak{L}_{\varphi
}^{(\mathfrak{g},V)}$.

$(ii)$ $\mathfrak{L}_{\varphi}^{(\mathfrak{g},V)}$ is trivial on $C_{i}$ if
and only if $[\varphi_{C_{i}}|_{C_{i}}]\neq0\in H^{1}(Y,O_{C_{i}}(C_{i}))$.

$(iii)$ When $V$ is the standard representation of $\mathfrak{g}$, there
exists a holomorphic fiberwise symmetric multi-linear form%
\[
f:\bigotimes^{r}\mathfrak{L}_{\varphi}^{(\mathfrak{g},V)}\longrightarrow
O_{Y}(D)
\]
with $r=0,2,3,4$ when $\mathfrak{g}=A_{n},D_{n},E_{6},E_{7}$ respectively such
that $\mathcal{E}_{\varphi}^{\mathfrak{g}}\cong aut_{0}(\mathfrak{L}_{\varphi
}^{(\mathfrak{g},V)},f)$.
\end{theorem}

When $V$ is a minuscule representation of $\mathfrak{g}$, there exists a
unique holomorphic structure on $\mathfrak{L}_{0}^{(\mathfrak{g}%
,V)}:=\bigoplus_{l}O(l)$ such that the action of $\mathcal{E}_{\varphi
}^{\mathfrak{g}}$ on this bundle is holomorphic and it can descend to $X$ as well.

\begin{example}
When we blowup 2 distinct points, we have a surface $Y$ with 2 $\left(
-1\right)  $-curves $l_{1}$ and $l_{2}$ as exceptional curves. $\mathfrak{L}%
_{0}\mathfrak{:=}O_{Y}\left(  l_{1}\right)  \oplus O_{Y}\left(  l_{2}\right)
$ is a $\mathbb{C}^{2}$-bundle and the bundle $\zeta_{0}^{A_{1}}$ of its
symmetries is a $sl\left(  2\right)  $- or $A_{1}$-bundle over $Y.$

When the 2 points become infinitesimally close, then $C_{1}=l_{2}-l_{1}$ is
effective, namely a $\left(  -2\right)  $-curve in $Y$. If we blow down
$C_{1}$ in $Y$, we get a surface $X$ with an $A_{1}$ singularity.
$\mathfrak{L}_{0}$ cannot descend to $X$ as $\mathfrak{L}_{0}|_{C_{1}}\cong
O_{\mathbb{P}^{1}}\left(  -1\right)  \oplus O_{\mathbb{P}^{1}}\left(
1\right)  $. Using the Euler sequence $0\rightarrow O_{\mathbb{P}^{1}%
}(-1)\rightarrow O_{\mathbb{P}^{1}}^{\oplus2}\rightarrow O_{\mathbb{P}^{1}%
}(1)\rightarrow0$, we deform $\mathfrak{L}_{0}|_{C_{1}}$ to become trivial and
using $p_{g}=0$ to lift this deformation to $Y$. The resulting bundles
$\mathfrak{L}_{\varphi}$ and $\zeta_{\varphi}^{A_{1}}$ do descend to $X$.
\end{example}

For every $\ ADE$ case with $V$ the standard representation, we have
$\mathfrak{L}_{0}^{(\mathfrak{g},V)}|_{C_{i}}\cong O_{\mathbb{P}^{1}}^{\oplus
m}+(O_{\mathbb{P}^{1}}(1)+O_{\mathbb{P}^{1}}(-1))^{\oplus n}$. For $A_{n}$
cases, our arguments are similar to the above $A_{1}$ case. For $D_{n}$ cases,
further arguments are needed as the pairs of $O_{\mathbb{P}^{1}}(\pm1)$ in
$\mathfrak{L}_{0}^{(D_{n},\mathbb{C}^{2n})}|_{C_{i}}$ are in different
locations comparing with the $A_{n}$ cases, and we also need to check the
holomorphic structure $\overline{\partial}_{\mathfrak{\varphi}}$ on
$\mathfrak{L}_{0}^{(D_{n},\mathbb{C}^{2n})}$ preserves the natural quadratic
form $q$. For the $E_{6}$ (resp. $E_{7}$) case, since the cubic form $c$
(resp. quartic form $t$) is more complicated than the quadratic form $q$ in
$D_{n}$ cases, the calculations are more involved. The $E_{8}$ case is rather
different and we handle it by reductions to $A_{7}$ and $D_{7}$ cases.

The organization of this paper is as follows. Section 2 gives the construction
of $ADE$ Lie algebra bundles over $Y$ directly. In section 3, we review the
definition of minuscule representations and construct all minuscule
representations using $(-1)$-curves in $Y$. Using these, we construct the Lie
algebra bundles and minuscule representation bundles which can descend to $X$
in $A_{n}$, $D_{n}$ and $E_{n}$ ($n\neq8$) cases separately in section 4, 5
and 6. The proofs of the main theorems in this paper are given in section 7.

Notations: for a holomorphic bundle $(E_{0},\overline{\partial}_{0})$, if we
construct a new holomorhic structure $\overline{\partial}_{\varphi}$ on
$E_{0}$, then we denote the resulting bundle as $E_{\varphi}$.

\bigskip

$\mathbf{Acknowledgements.}$ We are grateful to Jiajin Zhang for many useful
comments and discussions. The work of the second author was supported by a
direct allocation grant from the Research Grants Council of the Hong Kong
Special Administrative Region, China (reference No. 2060436).

\section{ $ADE$ Lie algebra bundles}

\subsection{ $ADE$ singularities}

A rational singularity $p$ in a surface $X$ can be described locally as a
quotient singularity $\mathbb{C}^{2}/\Gamma$ with $\Gamma$ a finite subgroup
of $SL(2,\mathbb{C)}$. It is also called a Kleinian singularity or $ADE$
singularity \cite{BPV}. We can write $\mathbb{C}^{2}/\Gamma$ as zeros of a
polynomial $F(X,Y,Z)$ in $\mathbb{C}^{3}$, where $F(X,Y,Z)\ $is $X^{n}+YZ$,
$X^{n+1}+XY^{2}+Z^{2}$, $X^{4}+Y^{3}+Z^{2}$, $X^{3}Y+Y^{3}+Z^{2}$ or
$X^{5}+Y^{3}+Z^{2}$ and the corresponding singularity is called of type
$A_{n}$, $D_{n}$, $E_{6}$, $E_{7}$ or $E_{8}$ respectively. The reason is if
we consider the minimal resolution $Y$ of $X$, then every irreducible
component of the exceptional divisor $C=\sum_{i=1}^{n}C_{i}$ is a smooth
rational curve with normal bundle $O_{\mathbb{P}^{1}}(-2)$, i.e. a
$(-2)$-curve, and the dual graph of the exceptional divisor is an $ADE$ Dynkin
diagram. The corresponding roots in the Dynkin diagrams are labelled as
follows:%
\[
\underset{\text{Figure 1. The root system of }A_{n}}%
{\setlength{\unitlength}{1.0cm}\begin{picture}(5, 2) \put(0, 1){\circle*{.2}} \put(1.2, 1){\circle*{.2}} \put(2.4, 1){\circle*{.2}} \put(3.6, 1){\circle*{.2}} \put(4.8, 1){\circle*{.2}} \put(0,1){\line(1, 0){1.2}} \put(1.5, 1){\circle*{.1}} \put(1.8, 1){\circle*{.1}}\put(2.1, 1){\circle*{.1}} \put(2.4, 1){\line(1, 0){1.2}} \put(3.6, 1){\line(1, 0){1.2}} \put(-0.5, .5){$C_1$} \put(1, .5){$C_2$} \put(2.3, .5){$C_{n-2}$} \put(3.5, .5){$C_{n-1}$} \put(4.6, .5){$C_{n}$} \end{picture}\ \ }%
\]

\[
\underset{\text{Figure 2. The root system of }D_{n}}%
{\setlength{\unitlength}{1.0cm}\begin{picture}(5, 3) \put(0, 1){\circle*{.2}} \put(1.2, 1){\circle*{.2}} \put(2.4,1){\circle*{.2}} \put(3.6,1){\circle*{.2}} \put(4.8,1){\circle*{.2}} \put(3.6,2.2){\circle*{.2}} \put(1.5, 1){\circle*{.1}} \put(1.8, 1){\circle*{.1}} \put(2.1, 1){\circle*{.1}} \put(0, 1){\line(1, 0){1.2}} \put(2.4, 1){\line(1, 0){1.2}} \put(3.6, 1){\line(1, 0){1.2}} \put(3.6, 1){\line(0, 1){1.2}} \put(-0.5, .5){$C_1$} \put(1, .5){$C_2$} \put(2.3, .5){$C_{n-3}$} \put(3.5, .5){$C_{n-2}$} \put(4.5, .5){$C_{n-1}$} \put(3.8, 2.1){$C_{n}$} \end{picture}\ }%
\]

\[
\underset{\text{Figure 3. The root system of }E_{n}}%
{\setlength{\unitlength}{1.0cm}\begin{picture}(6, 3) \put(0, 1){\circle*{.2}} \put(1.2, 1){\circle*{.2}} \put(2.4, 1){\circle*{.2}} \put(3.6, 1){\circle*{.2}} \put(4.8, 1){\circle*{.2}} \put(6, 1){\circle*{.2}} \put(3.6,2.2){\circle*{.2}} \put(1.5, 1){\circle*{.1}} \put(1.8, 1){\circle*{.1}} \put(2.1, 1){\circle*{.1}} \put(0, 1){\line(1, 0){1.2}} \put(2.4, 1){\line(1, 0){1.2}} \put(3.6, 1){\line(1, 0){1.2}} \put(4.8,1){\line(1,0){1.2}} \put(3.6, 1){\line(0, 1){1.2}} \put(-0.5, .5){$C_1$} \put(1.1, .5){$C_2$} \put(2.3, .5){$C_{n-4}$} \put(3.5, .5){$C_{n-3}$} \put(4.7, .5){$C_{n-2}$} \put(5.8, .5){$C_{n-1}$} \put(3.8, 2.1){$C_{n}$} \end{picture}\ \ }%
\]
%%%%%%%%%%%%%%%%%%%%%%%%%%%%%%%%%%%%%%%%%%%%%

There is a natural decomposition%
\[
H^{2}(Y,\mathbb{Z})=H^{2}(X,\mathbb{Z})\oplus\Lambda\text{,}%
\]
where $\Lambda=\{\sum a_{i}[C_{i}]|a_{i}\in\mathbb{Z}\}$. The set
$\Phi:=\{\alpha\in\Lambda|\alpha^{2}=-2\}$ is a simply-laced (i.e. $\ ADE$)
root system of a simple Lie algebra $\mathfrak{g}$ and $\Delta=\{[C_{i}]\}$ is
a base of $\Phi$. For any $\alpha\in\Phi$, there exists a unique divisor
$D=\sum a_{i}C_{i}$ with $\alpha=[D]$, and we define a line bundle
$O(\alpha):=O(D)$ over $Y$.

\subsection{Lie algebra bundles}

We define a Lie algebra bundle of type $\mathfrak{g}$ over $Y$ as follows:

\begin{center}
$\mathcal{E}_{0}^{\mathfrak{g}}:=O^{\oplus n}\oplus\bigoplus_{\alpha\in\Phi
}O(\alpha)$.
\end{center}

For every open chart $U$ of $Y$, we take $x_{\alpha}^{U}$ to be a nonvanishing
holomorphic section of $O_{U}(\alpha)$ and $h_{i}^{U}$ ($i=1,\cdots,n$)
nonvanishing holomorphic sections of $O_{U}^{\oplus n}$. Define a Lie algebra
structure $[,]$ on $\mathcal{E}_{0}^{\mathfrak{g}}$ such that $\{x_{\alpha
}^{U}$'s, $h_{i}^{U}$'s$\}$ is the Chevalley basis \cite{H}, i.e.

(a) $[h_{i}^{U},$ $h_{j}^{U}]=0$, $1\leq i$, $j\leq n$.

(b) $[h_{i}^{U},$ $x_{\alpha}^{U}]=\left\langle \alpha\text{, }C_{i}%
\right\rangle x_{\alpha}^{U}$, $1\leq i\leq n$, $\alpha\in\Phi$.

(c) $[x_{\alpha}^{U},$ $x_{-\alpha}^{U}]=h_{\alpha}^{U}$ is a $\mathbb{Z}%
$-linear combination of $h_{i}^{U}$.

(d) If $\alpha$, $\beta$ are independent roots, and $\beta-r\alpha
,\cdots,\beta+q\alpha$ is the $\alpha$-string through $\beta$, then
$[x_{\alpha}^{U},$ $x_{\beta}^{U}]=0$ if $q=0$, otherwise $[x_{\alpha}^{U},$
$x_{\beta}^{U}]=\pm(r+1)x_{\alpha+\beta}^{U}$.

Since $\mathfrak{g}$ is simply-laced, all its roots have the same length, we
have any $\alpha$-string through $\beta$ is of length at most $2$. So (d) can
be written as $[x_{\alpha}^{U},$ $x_{\beta}^{U}]=n_{\alpha,\beta}%
x_{\alpha+\beta}^{U}$, where $n_{\alpha,\beta}=\pm1$ if $\alpha+\beta\in\Phi$,
otherwise $n_{\alpha,\beta}=0$. From the Jacobi identity, we have for any
$\alpha,\beta,\gamma\in\Phi$, $n_{\alpha,\beta}n_{\alpha+\beta,\gamma
}+n_{\beta,\gamma}n_{\beta+\gamma,\alpha}+n_{\gamma,\alpha}n_{\gamma
+\alpha,\beta}=0$. This Lie algebra structure is compatible with different
trivializations of $\mathcal{E}_{0}^{\mathfrak{g}}$ \cite{LZ}.

By Friedman-Morgan \cite{FM}, a bundle over $Y$ can descend to $X$ if and only
if its restriction to each irreducible component $C_{i}$ of the exceptional
divisor is trivial. But $\mathcal{E}_{0}^{\mathfrak{g}}|_{C_{i}}$ is not
trivial as $O([C_{i}])|_{C_{i}}\cong O_{\mathbb{P}^{1}}(-2)$. We will
construct a new holomorphic structure on $\mathcal{E}_{0}^{\mathfrak{g}}$,
which preserves the Lie algebra structure and therefore the resulting bundle
$\mathcal{E}_{\varphi}^{\mathfrak{g}}$ can descend to $X$.

As we have fixed a base $\Delta$ of $\Phi$, we have a decomposition $\Phi
=\Phi^{+}\cup\Phi^{-}$ into positive and negative roots.

\begin{definition}
Given any $\varphi=(\varphi_{\alpha})_{\alpha\in\Phi^{+}}\in\Omega
^{0,1}(Y,\bigoplus_{\alpha\in\Phi^{+}}O(\alpha))$, we define $\overline
{\partial}_{\varphi}:\Omega^{0,0}(Y,\mathcal{E}_{0}^{\mathfrak{g}%
})\longrightarrow\Omega^{0,1}(Y,\mathcal{E}_{0}^{\mathfrak{g}})$ by
\[
\overline{\partial}_{\varphi}:=\overline{\partial}_{0}+ad(\varphi
):=\overline{\partial}_{0}+\sum_{\alpha\in\Phi^{+}}ad(\varphi_{\alpha
})\text{,}%
\]
where $\overline{\partial}_{0}$ is the standard holomorphic structure of
$\mathcal{E}_{0}^{\mathfrak{g}}$. More explicitly, if we write $\varphi
_{\alpha}=c_{\alpha}^{U}x_{\alpha}^{U}$ locally for some one form $c_{\alpha
}^{U}$, then $ad(\varphi_{\alpha})=c_{\alpha}^{U}ad(x_{\alpha}^{U})$.
\end{definition}

\begin{proposition}
\label{Lie bracket}$\overline{\partial}_{\varphi}$ is compatible with the Lie
algebra structure, i.e. $\overline{\partial}_{\mathcal{\varphi}}[,]=0$.

\begin{proof}
This follows directly from the Jacobi identity.
\end{proof}
\end{proposition}

For $\overline{\partial}_{\mathcal{\varphi}}$ to define a holomorphic
structure, we need
\[
0=\overline{\partial}_{\mathcal{\varphi}}^{2}=\sum_{\alpha\in\Phi^{+}%
}(\overline{\partial}_{0}c_{\alpha}^{U}+\sum_{\beta+\gamma=\alpha}%
(n_{\beta,\gamma}c_{\beta}^{U}c_{\gamma}^{U}))ad(x_{\alpha}^{U})\text{,}%
\]
that is $\overline{\partial}_{0}\varphi_{\alpha}+\sum_{\beta+\gamma=\alpha
}(n_{\beta,\gamma}\varphi_{\beta}\varphi_{\gamma})=0$ for any $\alpha\in
\Phi^{+}$. Explicitly:%
\[
\left\{
\begin{tabular}
[c]{cc}%
$\overline{\partial}_{0}\varphi_{C_{i}}=0$ & $i=1,2\cdots,n$\\
$\overline{\partial}_{0}\varphi_{C_{i}+C_{j}}=n_{C_{i},C_{j}}\varphi_{C_{i}%
}\varphi_{C_{j}}$ & $\text{ if }C_{i}+C_{j}\in\Phi^{+}$\\
$\vdots$ &
\end{tabular}
\ \right.
\]

\begin{proposition}
\label{holomorphic}Given any $(\varphi_{C_{i}})_{i=1}^{n}\in\Omega^{0,1}(Y,$
$\bigoplus_{i=1}^{n}O(C_{i}))$ with $\overline{\partial}_{0}\varphi_{C_{i}}=0$
for every $i$, it can be extended to $\varphi=(\varphi_{\alpha})_{\alpha
\in\Phi^{+}}\in\Omega^{0,1}(Y,\bigoplus_{\alpha\in\Phi^{+}}O(\alpha))$ such
that $\overline{\partial}_{\mathcal{\varphi}}^{2}=0$. Namely we have a
holomorphic vector bundle $\mathcal{E}_{\varphi}^{\mathfrak{g}}$ over $Y$.
\end{proposition}

To prove this proposition, we need the following lemma. For any $\alpha
=\sum_{i=1}^{n}a_{i}C_{i}\in\Phi^{+}$, we define $ht(\alpha):=\sum_{i=1}%
^{n}a_{i}$.

\begin{lemma}
\label{cohomology}For any $\alpha\in\Phi^{+}$, $H^{2}(Y,O(\alpha))=0$.

\begin{proof}
If $ht(\alpha)=1$, i.e. $\alpha=C_{i}$, $H^{2}(Y,O(C_{i}))=0$ follows from the
long exact sequence associated to $0\rightarrow O_{Y}\rightarrow O_{Y}%
(C_{i})\rightarrow O_{C_{i}}(C_{i})\rightarrow0$ and $p_{g}=0$.

By induction, suppose the lemma is true for every $\beta$ with $ht(\beta)=m$.
Given any $\alpha$ with $ht(\alpha)=m+1$, by Lemma A in $\S 10.2$ of \cite{H},
there exists some $C_{i}$ such that $\alpha\cdot C_{i}=-1$, i.e.
$\beta:=\alpha-C_{i}\in\Phi^{+}$ with $ht(\beta)=m$. Using the long exact
sequence associated to $0\rightarrow O_{Y}(\beta)\rightarrow O_{Y}%
(\alpha)\rightarrow O_{C_{i}}(\alpha)\rightarrow0$, $O_{C_{i}}(\alpha)\cong
O_{\mathbb{P}^{1}}(-1)$ and $H^{2}(Y,O(\beta))=0\ $by induction, we have
$H^{2}(Y,O(\alpha))=0$.
\end{proof}
\end{lemma}

\begin{proof}
$($of Proposition \ref{holomorphic}$)$\textit{\ We solve the equations
}$\overline{\partial}_{0}\varphi_{\alpha}=\sum_{\beta+\gamma=\alpha}%
n_{\beta,\gamma}\varphi_{\beta}\varphi_{\gamma}$\textit{\ for }$\varphi
_{\alpha}\in\Omega^{0,1}(Y,$\textit{\ }$O(\alpha))$\textit{\ inductively on
}$ht(\alpha)$\textit{.}

\textit{For }$ht(\alpha)=2\,$\textit{, i.e. }$\alpha=C_{i}+C_{j}%
$\textit{\ with }$C_{i}\cdot C_{j}=1$\textit{, since }$[\varphi_{C_{i}}%
\varphi_{C_{j}}]\in H^{2}(Y,O(C_{i}+C_{j}))=0$\textit{, we can find }%
$\varphi_{C_{i}+C_{j}}$\textit{\ satisfying }$\overline{\partial}_{0}%
\varphi_{C_{i}+C_{j}}=\pm\varphi_{C_{i}}\varphi_{C_{j}}$\textit{.}

\textit{Suppose we have solved the equations for all }$\varphi_{\beta}%
$'s\textit{\ with }$ht(\beta)\leq m$\textit{. For}%
\[
\overline{\partial}_{0}\varphi_{\alpha}=\sum_{\beta+\gamma=\alpha}%
n_{\beta,\gamma}\varphi_{\beta}\varphi_{\gamma}%
\]
\textit{ with }$ht(\alpha)=m+1$\textit{, we have }$ht(\beta),ht(\gamma)\leq
m$\textit{. Using }$\overline{\partial}_{0}(\sum_{\beta+\gamma=\alpha}%
n_{\beta,\gamma}\varphi_{\beta}\varphi_{\gamma})=\sum_{\delta+\lambda
+\mu=\alpha}(n_{\delta,\lambda}n_{\delta+\lambda,\mu}+n_{\lambda,\mu
}n_{\lambda+\mu,\delta}+n_{\mu,\delta}n_{\mu+\delta,\lambda})\varphi_{\delta
}\varphi_{\lambda}\varphi_{\mu}=0$, $[\sum_{\beta+\gamma=\alpha}%
n_{\beta,\gamma}\varphi_{\beta}\varphi_{\gamma}]\in H^{2}(Y,O(\alpha
))=0$\textit{, we can solve for }$\varphi_{\alpha}$\textit{.}
\end{proof}

\bigskip

Denote
\[
\Psi_{Y}\triangleq\{\varphi=(\varphi_{\alpha})_{\alpha\in\Phi^{+}}\in
\Omega^{0,1}(Y,\bigoplus_{\alpha\in\Phi^{+}}O(\alpha))|\overline{\partial
}_{\mathcal{\varphi}}^{2}=0\}\text{,}%
\]
and
\[
\Psi_{X}\triangleq\{\varphi\in\Psi_{Y}|[\varphi_{C_{i}}|_{C_{i}}]\neq0\text{
for }i=1,2,\cdots,n\}\text{.}%
\]

\begin{theorem}
\label{thm1}$\mathcal{E}_{\varphi}^{\mathfrak{g}}$ is trivial on $C_{i}$ if
and only if $[\varphi_{C_{i}}|_{C_{i}}]\neq0\in H^{1}(Y,O_{C_{i}}(C_{i}))$.

\begin{proof}
We will discuss the $ADE$ cases separately in $\S 4$, $\S 5$, $\S 6$ and the
proof will be completed in $\S 7$.
\end{proof}
\end{theorem}

The next lemma says that given any $C_{i}$, there always exists $\varphi
_{C_{i}}\in\Omega^{0,1}(Y,$ $O(C_{i}))$ such that $0\neq\lbrack\varphi_{C_{i}%
}|_{C_{i}}]\in H^{1}(Y,$ $O_{C_{i}}(C_{i}))\cong\mathbb{C}$.

\begin{lemma}
\label{surj}For any $C_{i}$ in $Y$, the restriction homomorphism $H^{1}(Y,$
$O_{Y}(C_{i}))\rightarrow H^{1}(Y,$ $O_{C_{i}}(C_{i}))$ is surjective.

\begin{proof}
The above restriction homomorphism is part of a long exact sequence induced by
$0\rightarrow O_{Y}\rightarrow O_{Y}(C_{i})\rightarrow O_{C_{i}}%
(C_{i})\rightarrow0$. The lemma follows directly from $\mathit{p}_{g}(Y)=0$.
\end{proof}
\end{lemma}

\section{Minuscule representations and $(-1)$-curves}

\subsection{Standard representations}

For$\ ADE$ Lie algebras, $A_{n}=sl\left(  n+1\right)  $ is the space of
tracefree endomorphisms of ${\mathbb{C}}^{n+1}$ and $D_{n}=o\left(  2n\right)
$ is the space of infinitesimal automorphisms of ${\mathbb{C}}^{2n}$ which
preserve a non-degenerate quadratic form $q$ on ${\mathbb{C}}^{2n}$. In fact,
$E_{6}$ (resp. $E_{7}$) is the space of infinitesimal automorphisms of
${\mathbb{C}}^{27}$ (resp. ${\mathbb{C}}^{56}$) which preserve a particular
cubic form $c$ on ${\mathbb{C}}^{27}$ (resp. quartic form $t$ on ${\mathbb{C}%
}^{56}$) \cite{A}. We call the above representation the \textit{standard
representation} of $\mathfrak{g}$, i.e.%
\[%
\begin{tabular}
[c]{|c|c|}\hline
$\mathfrak{g}$ & standard representation\\\hline
$A_{n}=sl\left(  n+1\right)  $ & ${\mathbb{C}}^{n+1}$\\\hline
$D_{n}=o\left(  2n\right)  $ & ${\mathbb{C}}^{2n}$\\\hline
$E_{6}$ & ${\mathbb{C}}^{27}$\\\hline
$E_{7}$ & ${\mathbb{C}}^{56}$\\\hline
\end{tabular}
\ \ \ \ \
\]
Note all these standard representations are the fundamental representations
corresponding to the left nodes (i.e. $C_{1}$) in the corresponding Dynkin
diagrams (Figure 1, 2 and 3) and they are minuscule representations.

\subsection{Minuscule representations}

\begin{definition}
A minuscule $($resp. quasi-minuscule$)$ representation of a semi-simple Lie
algebra is an irreducible representation such that the Weyl group acts
transitively on all the weights $($resp. non-zero weights$)$.
\end{definition}

Minuscule representations are always fundamental representations and
quasi-minuscule representations are either minuscule or adjoint representations.%

\[%
\begin{tabular}
[c]{|c|c|}\hline
$\mathfrak{g}$ & Miniscule representations\\\hline
$A_{n}=sl\left(  n+1\right)  $ & ${\wedge}^{k}{\mathbb{C}}^{n+1}$ for
$k=1,2,\cdots,n$\\\hline
$D_{n}=o\left(  2n\right)  $ &
\begin{tabular}
[c]{ccc}%
${\mathbb{C}}^{2n},$ & ${\mathcal{S}}^{+},$ & ${\mathcal{S}}^{-}$%
\end{tabular}
\\\hline
$E_{6}$ &
\begin{tabular}
[c]{cc}%
${\mathbb{C}}^{27},$ & $\overline{{\mathbb{C}}^{27}}$%
\end{tabular}
\\\hline
$E_{7}$ & ${\mathbb{C}}^{56}$\\\hline
\end{tabular}
\
\]
Note $E_{8}$ has no minuscule representation.

\subsection{Configurations of $(-1)$-curves}

In this subsection, we describe $(-1)$-curves in $X$ and $Y$.

\begin{definition}
A $(-1)$-curve in a surface $Y$ is a genus zero $($possibly reducible$)$ curve
$l$ in $Y$ with $l\cdot l=-1$.
\end{definition}

\begin{remark}
The genus zero condition can be replaced by $l\cdot K_{Y}=-1$ by the genus
formula, where $K_{Y}$ is the canonical divisor of $Y$.
\end{remark}

Let $C_{0}$ be a curve in $X$ passing through $p$.

\begin{definition}
$(i)$ $C_{0}$ is called a $(-1)$-curve in $X$ if there exists a $(-1)$-curve
$l$ in $Y$ such that $\pi(l)=C_{0}$, or equivalently the strict transform of
$C_{0}$ is a $(-1)$-curve $\widetilde{C_{0}}$ in $Y$. $(ii)$ The multiplicity
of $C_{0}$ at $p$ is defined to be $\sum_{i=1}^{n}a_{i}[C_{i}]\in\Lambda$,
where $a_{i}=\widetilde{C_{0}}\cdot C_{i}$.
\end{definition}

Recall from Lie theory, any irreducible representation of a simple Lie algebra
is determined by its lowest weight. The fundamental
representations\footnote{The usual definition for fundamental representations
uses highest weight. But in this paper, we will use lowest weight for
simplicity of notations.} are those irreducible representations whose lowest
weight is dual to the negative of some base root. If $C_{0}\subset X$ has
multiplicity $C_{k}$ at $p$ whose dual weight determines a minuscule
representation $V$, then we use $C_{0}^{k}$ to denote $\widetilde{C_{0}}$. The
construction of such $X$'s and $C_{0}$'s can be found in appendix.

\begin{definition}
\label{min}$(i)$ We call $C_{0}$ has minuscule multiplicity $C_{k}\in\Lambda$
at $p$ if $C_{0}$ has multiplicity $C_{k}$ and the dual weight of $-C_{k}$
determines a minuscule representation $V$. $(ii)$ In this case, we denote
$I^{(\mathfrak{g},V)}=\{l:(-1)$-curve in $Y|\pi(l)=C_{0}\}$.
\end{definition}

If there is no ambiguity, we will simply write $I^{(\mathfrak{g},V)}$ as $I$.
Note that $I\subset C_{0}^{k}+\Lambda_{\geq0}$, where $\Lambda_{\geq0}=\{\sum
a_{i}[C_{i}]:a_{i}\geq0\}$.

\begin{lemma}
\label{card}In the above situation, the cardinality of $I$ is given by
$|I|=\dim V$.

\begin{proof}
By the genus formula and every $C_{i}\cong\mathbb{P}^{1}$ being a
$(-2)$-curve, we have $C_{i}\cdot K_{Y}=0$. Since $C_{0}^{k}\cdot K_{Y}=-1$,
each $(-1)$-curve has the form $l=C_{0}^{k}+\sum a_{i}C_{i}$ with $a_{i}$'s
non-negative integers. From $l\cdot$ $l=-1,$ we can determine $\{a_{i}%
\}^{\prime}$s for $l$ to be a $(-1)$-curve by direct computations.
\end{proof}
\end{lemma}

\begin{remark}
The intersection product is negative definite on the sublattice of $Pic(X)$
generated by $C_{0}^{k},C_{1},\cdots,C_{n}$ and we use its negative as an
inner product.
\end{remark}

\begin{lemma}
\label{rep}In the above situation, for any $l\in I$, $\alpha\in\Phi$, we have
$|l\cdot\alpha|\leq1$.

\begin{proof}
We claim that for any $v\in C_{0}^{k}+\Lambda$, we have $v\cdot v\leq-1$. We
prove the claim by direct computations. In $(A_{n},{\wedge}^{k}{\mathbb{C}%
}^{n+1})$ case:

$(C_{0}^{k}+\sum a_{i}C_{i})^{2}$

$=-1+2a_{k}-(a_{1}^{2}+(a_{1}-a_{2})^{2}+\cdots+(a_{k-1}-a_{k})^{2}%
)-((a_{k}-a_{k+1})^{2}+\cdots+a_{n}^{2})$

$\leq-1$.

The other cases can be proven similarly.

Since $l,~l+\alpha,$ $l-\alpha\in C_{0}^{k}+\Lambda$ by assumptions, we have
$l\cdot l=-1\geq(l+\alpha)\cdot(l+\alpha)$, hence $l\cdot\alpha\leq1$. Also
$l\cdot l=-1\geq(l-\alpha)\cdot(l-\alpha)$, hence $l\cdot\alpha\geq-1$.
\end{proof}
\end{lemma}

\begin{lemma}
\label{lc}In the above situation, for any $l\in I$ which is not $C_{0}^{k}$,
there exists $C_{i}$ such that $l\cdot C_{i}=-1$.

\begin{proof}
From $l=C_{0}^{k}+\sum a_{i}C_{i}\neq C_{0}^{k}$ $(a_{i}\geq0)$, we have
$a_{k}\geq1$. From $l\cdot l=-1$, we have $(\sum a_{i}C_{i})^{2}=-2a_{k}$. If
there does not exist such an $i$ with $l\cdot C_{i}=-1$, then by Lemma
\ref{rep}, $l\cdot C_{i}\geq0$ for every $i$, $l\cdot(\sum a_{i}C_{i})\geq0$.
But $l\cdot(\sum a_{i}C_{i})=a_{k}+(\sum a_{i}C_{i})^{2}=-a_{k}\leq-1$ leads
to a contradiction.
\end{proof}
\end{lemma}

\begin{lemma}
\label{cohomology2}In the above situation, for any $l,l^{\prime}\in I$,
$H^{2}(Y,O(l-l^{\prime}))=0$.

\begin{proof}
Firstly, we prove $H^{2}(Y,O(C_{0}^{k}-l))=0$ for any $l=C_{0}^{k}+\sum
a_{i}C_{i}\in I$ inductively on $ht(l):=\sum a_{i}$. If $ht(l)=0$, i.e. $l$ is
$C_{0}^{k}$, the claim follows from $\mathit{p}_{g}=0$. Suppose the claim is
true for any $l^{\prime}\in I$ with $ht(l^{\prime})\leq m-1$. Then for any
$l\in I$ with $ht(l)=m$, by Lemma \ref{lc}, there exists $i$ such that $l\cdot
C_{i}=-1$. This implies $(l-C_{i})\in I$ with $ht(l-C_{i})=m-1$ and therefore
$H^{2}(Y,O(C_{0}^{k}-(l-C_{i})))=0$ by induction hypothesis. Using the long
exact sequence induced from%
\[
0\rightarrow O_{Y}(C_{0}^{k}-l)\rightarrow O_{Y}(C_{0}^{k}-(l-C_{i}%
))\rightarrow O_{C_{i}}(C_{0}^{k}-(l-C_{i}))\rightarrow0
\]
and $O_{C_{i}}(C_{0}^{k}-(l-C_{i}))\cong O_{\mathbb{P}^{1}}(-1)$ or
$O_{\mathbb{P}^{1}}$, we have the claim.

If $H^{2}(Y,O(l-l^{\prime}))\neq0$, then there exists a section $s\in
H^{0}(Y,K_{Y}(l^{\prime}-l))$ by Serre duality. Since there exists a nonzero
section $t\in H^{0}(Y,O(l-C_{0}^{k}))$, we have $st\in H^{0}(Y,K_{Y}%
(l^{\prime}-C_{0}^{k}))\cong H^{2}(Y,O(C_{0}^{k}-l^{\prime}))=0$, which is a contradiction.
\end{proof}
\end{lemma}

\subsection{Minuscule representations from $(-1)$-curves}

Recall from the $ADE$ root system $\Phi$, we can recover the corresponding Lie
algebra $\mathfrak{g}=\mathfrak{h}\oplus{\bigoplus_{\alpha\in\Phi}%
}\mathfrak{g}_{\alpha}$. As before, we use $\{x_{\alpha}$'s, $h_{i}$'s$\}$ to
denote its Chevalley basis. If $C_{0}$ has minuscule multiplicity $C_{k}$, we
denote%
\[
V_{0}:=\mathbb{C}^{I}={\bigoplus_{l\in I}}\mathbb{C}\langle v_{l}%
\rangle\text{,}%
\]
where $v_{l}$ is the base vector of $V_{0}$ generated by $l$. Then we define a
bilinear map $[,]:\mathfrak{g}\otimes V_{0}\rightarrow V_{0}$ (possibly up to
$\pm$ signs) as follows:

\begin{center}
$[x,v_{l}]=\left\{
\begin{array}
[c]{cc}%
\langle x,l\rangle v_{l} & \text{if }x\in\mathfrak{h}\\
\pm v_{l+\alpha} & \text{if }x=x_{\alpha},~l+\alpha~\in I\\
0 & \text{if }x=x_{\alpha},~l+\alpha~\notin I
\end{array}
\right.  $
\end{center}

\begin{proposition}
\label{representation}The signs in the above bilinear map $\mathfrak{g}\otimes
V_{0}\rightarrow V_{0}$ can be chosen so that it defines an action of
$\mathfrak{g}$ on $V_{0}$. Moreover, $V_{0}$ is isomorphic to the minuscule
representation $V$.

\begin{proof}
For the first part, similar to $\cite{L}$, we use Lemma \ref{rep} to show
$[[x,y],v_{l}]=[x,[y,v_{l}]]-[y,[x,v_{l}]]$.

For the second part, since $[x_{\alpha},v_{C_{0}^{k}}]=0$ for any $\alpha
\in\Phi^{-}$, $v_{C_{0}^{k}}$ is the lowest weight vector of $V_{0}$ with
weight corresponding to $-C_{k}$. Also we know the fundamental representation
$V$ corresponding to $-C_{k}$ has the same dimension with $V_{0}$ by lemma
\ref{card}. Hence $V_{0}$ is isomorphic to the minuscule representation $V$.
\end{proof}
\end{proposition}

Here we show how to determine the signs. Take any $l\in I$, $v_{l}$ is a
weight vector of the above action. For $x=x_{\alpha}$ and $v_{l}$ with weight
$w$, we define $[x,v_{l}]=n_{\alpha,w}v_{l+\alpha}$, where $n_{\alpha,w}=\pm1$
if $l+\alpha\in I$, otherwise $n_{\alpha,w}=0$. By $[[x,y],v_{l}%
]=[x,[y,v_{l}]]-[y,[x,v_{l}]]$, we have $n_{\alpha,\beta}n_{\alpha+\beta
,w}-n_{\beta,w}n_{\alpha,\beta+w}+n_{\alpha,w}n_{\beta,\alpha+w}=0$.

\begin{remark}
Recall for any $l=C_{0}^{k}+\sum a_{i}C_{i}\in I$, we define $ht(l):=\sum
a_{i}$. Using this, we can define a filtered structure for $I:I=I_{0}\supset
I_{1}\supset\cdots\supset I_{m}$, where $m=\max_{l\in I}ht(l)$, $I_{i}=\{l\in
I|ht(l)\leq m-i\}$ and $I_{i}\backslash I_{i+1}=\{l\in I|ht(l)=m-i\}$. This
$ht(l)$ also enables us to define a partial order of $I$. Say $|I|=N$, we
denote $l_{N}:=C_{0}^{k}$ since it is the only element with $ht=0$. Similarly,
$l_{N-1}:=C_{0}^{k}+C_{k}$. Of course, there are some ambiguity of this
ordering, if so, we will just make a choice to order these $(-1)$-curves.
\end{remark}

\subsection{Bundles from $(-1)$-curves}

The geometry of $(-1)$-curves in $Y$ can be used to construct representation
bundles of $\mathcal{E}_{\varphi}^{\mathfrak{g}}$ for every minuscule
representation of $\mathfrak{g}$. The proofs of theorems in this subsection
will be given in $\S 7$.

When $C_{0}\subset X$ has minuscule multiplicity $C_{k}$ at $p$ with the
corresponding minuscule representation $V$, we define$\footnote{When $X$ is a
del Pezzo surface, we use lines in $X$ to construct bundles [FM]. So here we
use $(-1)$-curves in $X$ to construct bundles.}$

\begin{center}
$\mathfrak{L}_{0}^{(\mathfrak{g},V)}:=\bigoplus_{l\in I^{(\mathfrak{g},V)}%
}O(l)$.
\end{center}

$\mathfrak{L}_{0}^{(\mathfrak{g},V)}$ has a natural filtration $F^{\bullet}$:
$\mathfrak{L}_{0}^{(\mathfrak{g},V)}=F^{0}\mathfrak{L}\supset F^{1}%
\mathfrak{L}\supset\cdots\supset F^{m}\mathfrak{L}$, induced from the
flittered structure on $I$, namely $F^{i}\mathfrak{L}_{0}^{(\mathfrak{g}%
,V)}=\bigoplus_{l\in I_{i}}O(l)$.

$\mathfrak{L}_{0}^{(\mathfrak{g},V)}$ can not descend to $X$ as $O_{C_{k}%
}(C_{0}^{k})\cong O_{\mathbb{P}^{1}}(1)$ (because $C_{k}\cdot C_{0}^{k}=1$ by
the definition of the minuscule multiplicity). For any $C_{i}$ and any $l\in
I$, we have $O_{C_{i}}(l)\cong O_{\mathbb{P}^{1}}(\pm1)$ or $O_{\mathbb{P}%
^{1}}$ by Lemma \ref{rep}. For every fixed $C_{i}$, if there is a $l\in I$
such that $O_{C_{i}}(l)\cong O_{\mathbb{P}^{1}}(1)$, then $(l+C_{i}%
)^{2}=-1=(l+C_{i})\cdot K_{Y}$, i.e. $l+C_{i}\in I$, also $O_{C_{i}}%
(l+C_{i})\cong O_{\mathbb{P}^{1}}(-1)$. That means among the direct summands
of $\mathfrak{L}_{0}^{(\mathfrak{g},V)}|_{C_{i}}$, $O_{\mathbb{P}^{1}}%
(1)\ $and $O_{\mathbb{P}^{1}}(-1)$ occur in pairs, and each pair is given by
two $(-1)$-curves in $I$ whose difference is $C_{i}$. This gives us a chance
to deform $\mathfrak{L}_{0}^{(\mathfrak{g},V)}$ to get another bundle which
can descend to $X$.

\begin{theorem}
\label{thm2}If there exists a $(-1)$-curve $C_{0}$ in $X$ with minuscule
multiplicity $C_{k}$ at $p$ and $\rho:\mathfrak{g}\longrightarrow End(V)$ is
the corresponding representation, then%
\[
(\mathfrak{L}_{\varphi}^{(\mathfrak{g},V)}:=\bigoplus_{l\in I}O(l),\text{
}\overline{\partial}_{\mathfrak{\varphi}}:=\overline{\partial}_{0}%
+\rho(\varphi))
\]
with $\varphi\in\Psi_{Y}$ is a holomorphic bundle over $Y$ which preserves the
filtration on $\mathfrak{L}_{0}^{(\mathfrak{g},V)}$ and it is a holomorphic
representation bundle of $\mathcal{E}_{\varphi}^{\mathfrak{g}}$. Moreover,
$\mathfrak{L}_{\varphi}^{(\mathfrak{g},V)}$ is trivial on $C_{i}$ if and only
if $[\varphi_{C_{i}}|_{C_{i}}]\neq0\in H^{1}(Y,O_{C_{i}}(C_{i}))$.
\end{theorem}

For $C_{k}$ with $k=1$, the corresponding minuscule representation $V$ is the
standard representation of $\mathfrak{g}$. When $\mathfrak{g}=A_{n}$, it is
simply $sl\left(  n+1\right)  =aut_{0}(V)$. When $\mathfrak{g}=D_{n}$ (resp.
$E_{6}$ and $E_{7}$), there exists a quadratic (resp. cubic and quartic) form
$f$ on $V$ such that $\mathfrak{g}=aut(V,f)$. The next theorem tells us that
we can globalize this construction over $Y$ to recover the Lie algebra bundle
$\mathcal{E}_{\varphi}^{\mathfrak{g}}$ over $Y$. But this does not work for
$\mathcal{E}_{\varphi}^{E_{8}}$ as $E_{8}$ has no standard representation.

\begin{theorem}
\label{thm3}Under the same assumptions as in theorem \ref{thm2} with $k=1$,
there exists a holomorphic fiberwise symmetric multi-linear form%
\[
f:\bigotimes^{r}\mathfrak{L}_{\varphi}^{(\mathfrak{g},V)}\longrightarrow
O_{Y}(D)
\]
with $r=0,2,3,4$ when $\mathfrak{g}=A_{n},D_{n},E_{6},E_{7}$ respectively such
that $\mathcal{E}_{\varphi}^{\mathfrak{g}}\cong aut_{0}(\mathfrak{L}_{\varphi
}^{(\mathfrak{g},V)},f)$.
\end{theorem}

It is obvious that $\mathcal{E}_{\varphi}^{\mathfrak{g}}$ does not depend on
the existence of the $(-1)$-curve $C_{0}$, for the minuscule representation
bundles, we have the following results.

\begin{theorem}
\label{thm4}There exists a divisor $B$ in $Y$ and an integer $k$, such that
the bundle $\mathbb{L}_{\mathfrak{\varphi}}^{(\mathfrak{g},V)}:=S^{k}%
\mathfrak{L}_{\mathfrak{\varphi}}^{(\mathfrak{g},V)}\otimes O(-B)$ with
$\varphi\in\Psi_{X}$ can descend to $X$ and does not depend on the existence
of $C_{0}$.
\end{theorem}

\subsection{Outline of Proofs for $\mathfrak{g\neq}E_{8}$}

When $\mathfrak{g\neq}E_{8}$, there exists a natural symmetric tensor $f$ on
its standard representation $V$ such that $\mathfrak{g}=aut_{0}(V,f)$. The set
$I^{(\mathfrak{g},\ V)}$ of $(-1)$-curves has cardinality $N=\dim V$. Given
$\mathfrak{\eta:=}(\mathfrak{\eta}_{i,j})_{N\times N}$ with $\mathfrak{\eta
}_{i,j}\in\Omega^{0,1}(Y,O(l_{i}-l_{j}))$ for every $l_{i}\neq l_{j}\in
I^{(\mathfrak{g},\ V)}$, we consider the operator $\overline{\partial
}_{\mathfrak{\eta}}:=\overline{\partial}_{0}+\eta$ on $\mathfrak{L}%
_{0}^{(\mathfrak{g},\ V)}:=\bigoplus_{l\in I^{(\mathfrak{g},\ V)}}O_{Y}(l)$.
We will look for $\eta$ which satisfy:

(1) (filtration) $\mathfrak{\eta}_{i,j}=0$ for $i>j$ for the partial ordering
introduced in $\S 3.4$.

(2) (holomorphic structure) $(\overline{\partial}_{0}+\mathfrak{\eta)}^{2}=0$.

(3) (Lie algebra structure) $\overline{\partial}_{\mathfrak{\eta}}f=0$.

(4) (descendent) For every $C_{k}$, if $l_{i}-l_{j}=C_{k}$, then $0\neq
\lbrack\mathfrak{\eta}_{i,j}|_{C_{k}}]\in H^{1}(Y,O_{C_{k}}(C_{k}))$.

\begin{remark}
Property $(2)$ implies that we can define a new holomorphic structure on
$\mathfrak{L}_{0}^{(\mathfrak{g},\ V)}$. Properties $(1)$ and $(3)$ require
that for any $\mathfrak{\eta}_{i,j}\neq0$, $\mathfrak{\eta}_{i,j}\in
\Omega^{0,1}(Y,O(\alpha))$ for some $\alpha\in\Phi^{+}$. We will show that if
$\mathfrak{\eta}$ satisfies (1), (2) and (3), then (4) is equivalent to
$\mathfrak{L}_{\mathfrak{\eta}}^{(\mathfrak{g},\ V)}$ being trivial on every
$C_{k}$, i.e. $\mathfrak{L}_{\mathfrak{\eta}}^{(\mathfrak{g},\ V)}$ can
descend to $X$.
\end{remark}

Denote
\[
\Xi_{Y}^{\mathfrak{g}}\triangleq\{\mathfrak{\eta=}(\mathfrak{\eta}%
_{i,j})_{N\times N}|\mathfrak{\eta}\text{ satisfies (1), (2) and
(3)}\}\text{,}%
\]
and%
\[
\Xi_{X}^{\mathfrak{g}}\triangleq\{\mathfrak{\eta\in}\Xi_{Y}^{\mathfrak{g}%
}|\mathfrak{\eta}\text{ satisfies (4)}\}\text{,}%
\]
then each $\mathfrak{\eta}$ in $\Xi_{Y}^{\mathfrak{g}}$ determines a filtered
holomorphic bundle $\mathfrak{L}_{\mathfrak{\eta}}^{(\mathfrak{g},\ V)}$ over
$Y$ together with a holomorphic tensor $f$ on it. It can descend to $X$ if
$\mathfrak{\eta}\in\Xi_{X}^{\mathfrak{g}}$.

\bigskip

Since $\mathfrak{g}=aut(V,f)$, for any $\mathfrak{\eta\in}\Xi_{Y}%
^{\mathfrak{g}}$, we have a holomorphic Lie algebra bundle $\mathcal{\zeta
}_{\mathfrak{\eta}}^{\mathfrak{g}}:=aut(\mathfrak{L}_{\mathfrak{\eta}%
}^{(\mathfrak{g},\ V)},f)$ over $Y$ of type $\mathfrak{g}$, and $\mathfrak{L}%
_{\mathfrak{\eta}}^{(\mathfrak{g},\ V)}$ is automatically a representation
bundle of $\mathcal{\zeta}_{\mathfrak{\eta}}^{\mathfrak{g}}$. Furthermore, if
$\mathfrak{\eta\in}\Xi_{X}^{\mathfrak{g}}$, then $\mathcal{\zeta
}_{\mathfrak{\eta}}^{\mathfrak{g}}$ can descend to $X$.

For a general minuscule representation of $\mathfrak{g}$, given any
$\mathfrak{\eta\in}\Xi_{Y}^{\mathfrak{g}}$, we show that there exists a unique
holomorphic structure on $\mathfrak{L}_{0}^{(\mathfrak{g},V)}$, such that the
action of $\mathcal{\zeta}_{\mathfrak{\eta}}^{\mathfrak{g}}$ on the new
holomorphic bundle $\mathfrak{L}_{\mathfrak{\eta}}^{(\mathfrak{g},V)}$ is
holomorphic. Furthermore, if $\mathfrak{\eta\in}\Xi_{X}^{\mathfrak{g}}$, then
$\mathfrak{L}_{\mathfrak{\eta}}^{(\mathfrak{g},V)}$ can descend to $X$.

\section{$A_{n}$ case}

We recall that $A_{n}=sl(n+1,\mathbb{C)=}aut_{0}(\mathbb{C}^{n+1})$ (where
$aut_{0}$ means tracefree endomorphisms). The standard representation of
$A_{n}$ is $\mathbb{C}^{n+1}$ and minuscule representations of $A_{n}$ are
$\wedge^{k}\mathbb{C}^{n+1}$, $k=1,2,\cdots,n$.

\subsection{$A_{n}$ standard representation bundle $\mathfrak{L}%
_{\mathfrak{\eta}}^{(A_{n},\mathbb{C}^{n+1})}$}

We consider a surface $X$ with an $A_{n}$ singularity $p$ and a $(-1)$-curve
$C_{0}$ passing through $p$ with multiplicity $C_{1}$, then $I^{(A_{n}%
,\mathbb{C}^{n+1})}=\{C_{0}^{1}+\sum_{i=1}^{k}C_{i}|$ $0\leq k\leq n\}$ has
cardinality $n+1$. We order these $(-1)$-curves: $l_{k}=C_{0}^{1}+\sum
_{i=1}^{n+1-k}C_{i}$ for $1\leq k\leq n+1$. For any $l_{i}\neq l_{j}\in I$,
$l_{i}\cdot l_{j}=0$. Fix any $C_{i}$, we have%
\[
l_{k}\cdot C_{i}=\left\{
\begin{tabular}
[c]{cc}%
$1,$ & $k=n+2-i$\\
$-1,$ & $k=n+1-i$\\
$0,$ & $\text{otherwise.}$%
\end{tabular}
\ \ \ \right.
\]

Define $\mathfrak{L}_{0}^{(A_{n},\mathbb{C}^{n+1})}:=\bigoplus_{l\in I}O(l)$
over $Y$, for simplicity, we write it as $\mathfrak{L}_{0}^{A_{n}}$.
$\mathfrak{L}_{0}^{A_{n}}$ can not descend to $X$, since for any $C_{i}$,%
\[
\mathfrak{L}_{0}^{A_{n}}|_{C_{i}}\cong O_{\mathbb{P}^{1}}^{\oplus(n-1)}\oplus
O_{\mathbb{P}^{1}}(1)\oplus O_{\mathbb{P}^{1}}(-1)\text{.}%
\]

Our aim is to find a new holomorphic structure on $\mathfrak{L}_{0}^{A_{n}}$
such that the resulting bundle can descend to $X$. First, we define
$\overline{\partial}_{\mathfrak{\eta}}:\Omega^{0,0}(Y,\mathfrak{L}_{0}^{A_{n}%
})\longrightarrow\Omega^{0,1}(Y,\mathfrak{L}_{0}^{A_{n}})$ on $\mathfrak{L}%
_{0}^{A_{n}}=\bigoplus_{k=1}^{n+1}O(l_{k})$ as follows:%
\[
\overline{\partial}_{\mathfrak{\eta}}=\left(
\begin{array}
[c]{cccc}%
\overline{\partial} & \eta_{1,2} & \cdots & \eta_{1,n+1}\\
0 & \overline{\partial} & \cdots & \eta_{2,n+1}\\
\vdots & \vdots & \ddots & \vdots\\
0 & 0 & \cdots & \overline{\partial}%
\end{array}
\right)
\]
where $\eta_{i,j}\in\Omega^{0,1}(Y,O(l_{i}-l_{j}))$ for any $j>i$. When $j>i$,
$l_{i}-l_{j}\in\Lambda$ is a positive root because of $l_{i}\cdot l_{j}=0$ and
our ordering of $l_{k}$'s.

The integrability condition $\overline{\partial}_{\mathfrak{\eta}}^{2}=0$ is
equivalent to, for $i=1,2,\cdots,n$,%
\[
\left\{
\begin{array}
[c]{l}%
\overline{\partial}\eta_{i,i+1}=0,\\
\overline{\partial}\eta_{i,j}=-\sum_{m=i+1}^{j-1}\mathfrak{\eta}%
_{i,m}\mathfrak{\cdot\mathfrak{\eta}}_{m,j}\text{, }~j\geq i+2\text{,}%
\end{array}
\right.
\]

Note $\eta_{i,j}\in\Omega^{0,1}(Y,O(l_{i}-l_{j}))=\Omega^{0,1}(Y,O(\alpha))$
for some $\alpha\in\Phi^{+}$. From%
\[
\sum_{m=i+1}^{j-1}[\mathfrak{\eta}_{i,m}\mathfrak{\cdot\mathfrak{\eta}}%
_{m,j}]\in H^{2}(Y,O(l_{i}-l_{j}))=0\text{,}%
\]
we can find $\eta_{i,j}$, such that $\overline{\partial}\eta_{i,j}%
=-\sum_{m=i+1}^{j-1}\mathfrak{\eta}_{i,m}\mathfrak{\cdot\mathfrak{\eta}}%
_{m,j}$. That is

\begin{proposition}
\label{holo2}Given any $\eta_{i,i+1}\in\Omega^{0,1}(Y,O(l_{i}-l_{i+1}))$ with
$\overline{\partial}\eta_{i,i+1}=0$ for $i=1,2,\cdots n$, there exists
$\eta_{i,j}\in\Omega^{0,1}(Y,O(l_{i}-l_{j}))$ for every $j>i$ such that
$\overline{\partial}_{\mathfrak{\eta}}$ defines a holomorphic structure on
$\mathfrak{L}_{0}^{A_{n}}$, i.e. $\overline{\partial}_{\mathfrak{\eta}}^{2}=0$.
\end{proposition}

We want to prove that there exists $\mathfrak{\eta\in}\Xi_{Y}^{A_{n}}$ such
that $\mathfrak{L}_{\mathfrak{\eta}}^{A_{n}}$ can descend to $X$, i.e.
$\mathfrak{L}_{\mathfrak{\eta}}^{A_{n}}|_{C_{i}}$ is trivial for every $C_{i}%
$. To prove this, we will construct $n+1$ holomorphic sections of
$\mathfrak{L}_{\mathfrak{\eta}}^{A_{n}}|_{C_{i}}$ which are linearly
independent everywhere on $C_{i}$. The following lemma will be needed for all
the $ADE$ cases.

\begin{lemma}
\label{section}Consider a vector bundle $(\mathfrak{L}:=\bigoplus_{i=1}%
^{N}O(l_{i}),$ $\overline{\partial}_{\mathfrak{L}}=\overline{\partial}%
_{0}+(\eta_{i,j})_{N\times N})$ over $Y$ with $\eta_{i,j}=0$ whenever $i\geq
j$. Suppose $C$ is a smooth $(-2)$-curve in $Y$ with $H^{1}(C,O_{C}(l_{i}))=0$
for every $i=1,2,\cdots N$, then for any fixed $i$ and any $s_{i}\in
H^{0}(C,O_{C}(l_{i}))$, the following equation for $s_{1}$, $s_{2}$, $\cdots$
$s_{i-1}$ has a solution,%
\[
\left(
\begin{array}
[c]{cccccc}%
\overline{\partial} & \eta_{1,2}|_{C} & \eta_{1,3}|_{C} & \cdots & \cdots &
\eta_{1,N}|_{C}\\
0 & \overline{\partial} & \eta_{2,3}|_{C} & \cdots & \cdots & \eta_{2,N}%
|_{C}\\
0 & 0 & \overline{\partial} & \cdots & \cdots & \eta_{3,N}|_{C}\\
\vdots & \vdots & \vdots & \ddots & \ddots & \vdots\\
\vdots & \vdots & \vdots & \ddots & \ddots & \vdots\\
0 & 0 & 0 & \cdots & \cdots & \overline{\partial}%
\end{array}
\right)  \left(
\begin{array}
[c]{c}%
s_{1}\\
\vdots\\
s_{i}\\
0\\
\vdots\\
0
\end{array}
\right)  =0\text{.}%
\]

\begin{proof}
The above equation is equivalent to:%
\begin{equation}
\overline{\partial}s_{i}=0\text{,} \tag{ $1$}%
\end{equation}%
\begin{equation}
\eta_{i-1,i}s_{i}+\overline{\partial}s_{i-1}=0\text{,} \tag{ $2$}%
\end{equation}%
\[
\vdots
\]%
\begin{equation}
\eta_{1,i}s_{i}+\cdots+\eta_{1,2}s_{2}+\overline{\partial}s_{1}=0\text{.}
\tag{ $i$}%
\end{equation}

Equation (1) is automatic as $s_{i}\in H^{0}(C,O_{C}(l_{i}))$. For equation
(2), since $\overline{\partial}\eta_{i-1,i}=0$ and $\overline{\partial}%
s_{i}=0$, we have $[\eta_{i-1,i}s_{i}]\in H^{1}(C,O_{C}(l_{i-1}))=0$, hence we
can find $s_{i-1}$ satisfying $\overline{\partial}s_{i-1}=-\eta_{i-1,i}s_{i}$.

Inductively, suppose we have found $s_{i},\cdots,s_{j-1}$ for the first
$(i-j)$ equations, then for the $(i-j+1)$-th equation: $\eta_{j,i}s_{i}%
+\cdots+\eta_{j,j+1}s_{j+1}+\overline{\partial}s_{j}=0$, we have
\[
\eta_{j,i}s_{i}+\cdots+\eta_{j,j+1}s_{j+1}\in\Omega^{0,1}(C,O_{C}%
(l_{j}))\text{.}%
\]
From $\overline{\partial}_{\mathfrak{L}}^{2}=0$, we have
\[
\overline{\partial}\eta_{k,m}=-(\eta_{k,k+1}\cdot\eta_{k+1,m}+\eta
_{k,k+2}\cdot\eta_{k+2,m}+\cdots+\eta_{k,m-1}\cdot\eta_{m-1,m})\text{.}%
\]
Then
\[
\overline{\partial}(s_{m})=-(\eta_{m,m+1}s_{m+1}+\cdots+\eta_{m,i}s_{i})
\]
implies
\[
\overline{\partial}(\eta_{j,i}s_{i}+\cdots+\eta_{j,j+1}s_{j+1})=0
\]
Therefore $[\eta_{j,i}s_{i}+\cdots+\eta_{j,j+1}s_{j+1}]\in H^{1}(C,O_{C}%
(l_{j}))=0$, hence we can find $s_{j}$ such that $\overline{\partial}%
s_{j}=-(\eta_{j,i}s_{i}+\cdots+\eta_{j,j+1}s_{j+1})$.
\end{proof}
\end{lemma}

Let us recall a standard result which says that the only non-trivial extension
of $O_{\mathbb{P}^{1}}(1)$ by $O_{\mathbb{P}^{1}}(-1)$ is the trivial bundle.
We will give an explicit construction of this trivialization as we will need a
generalization of it later.

\begin{lemma}
\label{indep}For an exact sequence over $\mathbb{P}^{1}:0\rightarrow
O_{\mathbb{P}^{1}}(-1)\rightarrow E\rightarrow O_{\mathbb{P}^{1}%
}(1)\rightarrow0$, the bundle $E$ is determined by the extension class
$[\varphi]\in Ext_{\mathbb{P}^{1}}^{1}(O(1),O(-1))\cong\mathbb{C}$ up to a
scalar multiple. If $[\varphi]\neq0$, $E$ is trivial, namely there exists two
holomorphic sections for $E$ which are linearly independent at every point in
$\mathbb{P}^{1}$.

\begin{proof}
With respect to the $($topological$)$ splitting $E=O_{\mathbb{P}^{1}%
}(-1)\oplus O_{\mathbb{P}^{1}}(1)$, the holomorphic structure on $E$ is given
by
\[
\overline{\partial}_{E}=\left(
\begin{array}
[c]{cc}%
\overline{\partial} & \varphi\\
0 & \overline{\partial}%
\end{array}
\right)
\]
with $\varphi\in Ext_{\mathbb{P}^{1}}^{1}(O(1),O(-1)))$. Let $t_{1},t_{2}\ $be
a base of $H^{0}(\mathbb{P}^{1},O(1))\cong\mathbb{C}^{2}$. Since $[\varphi
t_{i}]\in H^{1}(\mathbb{P}^{1},O(-1))=0$, we can find $u_{1},u_{2}\in
\Omega^{0}(\mathbb{P}^{1},O(-1))$, such that
\[
\left(
\begin{array}
[c]{cc}%
\overline{\partial} & \varphi\\
0 & \overline{\partial}%
\end{array}
\right)  \cdot\left(
\begin{array}
[c]{c}%
u_{i}\\
t_{i}%
\end{array}
\right)  =0\text{,}%
\]
i.e. $s_{1}=(u_{1},t_{1})^{t}$ and $s_{2}=(u_{2},t_{2})^{t}$ are two
holomorphic sections of $E$. Explicitly, we can take $s_{1}=(\frac
{1}{1+|z|^{2}},z)^{t}$, $s_{2}=(\frac{-\overline{z}}{1+|z|^{2}},1)^{t}$ in the
coordinate chart $\mathbb{C\subset P}^{1}$. It can be checked that $s_{1}$ and
$s_{2}$ are linearly independent over $\mathbb{P}^{1}$.
\end{proof}
\end{lemma}

From the above lemma, we have the following result.

\begin{lemma}
\label{indep3}Under the same assumption as in Lemma \ref{section}. Suppose
$\mathfrak{L}|_{C}\cong O_{\mathbb{P}^{1}}^{\oplus m}\oplus(O_{\mathbb{P}^{1}%
}(1)\oplus O_{\mathbb{P}^{1}}(-1))^{\oplus n}$ with each pair of
$O_{\mathbb{P}^{1}}(\pm1)$ corresponding to two $(-1)$-curves $l_{i}$ and
$l_{i+1}$ with $l_{i}-l_{i+1}=C$. Then $\mathfrak{L}|_{C}$ is trivial if and
only if $[\mathfrak{\eta}_{i,i+1}|_{C}]\neq0$ for every $\mathfrak{\eta
}_{i,i+1}\in\Omega^{0,1}(Y,O(C))$.

\begin{proof}
For simplicity, we assume $m=n=1$ and $O_{C}(l_{1})\cong O_{\mathbb{P}^{1}},$
$O_{C}(l_{2})\cong O_{\mathbb{P}^{1}}(-1)$, $O_{C}(l_{3})\cong O_{\mathbb{P}%
^{1}}(1)$ with $l_{2}-l_{3}=C$. If $[\mathfrak{\eta}_{2,3}|_{C}]\neq0$, by
Lemma \ref{section} and Lemma \ref{indep}, there exists two holomorphic
sections for $\mathfrak{L}|_{C}$ which are linearly independent at every point
in $C$: $s_{1}=(x_{1},u_{1},t_{1})^{t}$ and $s_{2}=(x_{2},u_{2},t_{2})^{t}$
with $u_{1},t_{1},u_{2},t_{2}$ given in the proof of Lemma \ref{indep}. By
$H^{0}(Y,O_{C}(l_{1}))\cong H^{0}(\mathbb{P}^{1},O)\cong\mathbb{C}$, there
exists one holomorphic section for $\mathfrak{L}|_{C}$ which is nowhere zero
on $C$: $s_{3}=(x_{3},0,0)^{t}$. These $s_{1},s_{2},s_{3}$ give a
trivialization of $\mathfrak{L}|_{C}$. If $[\mathfrak{\eta}_{2,3}|_{C}]=0$,
then $\mathfrak{L}_{N}|_{C}$ is an extension of $O_{\mathbb{P}^{1}}(1)\oplus
O_{\mathbb{P}^{1}}(-1)$ by $O_{\mathbb{P}^{1}}$ and there is no such
nontrivial extension.
\end{proof}
\end{lemma}

\begin{proposition}
\label{A1}The bundle $\mathfrak{L}_{\mathfrak{\eta}}^{A_{n}}$ over $Y$ with
$\mathfrak{\eta\in}\Xi_{Y}^{A_{n}}$ can descend to $X$ if and only if
$0\neq\lbrack\eta_{n+1-i,n+2-i}|_{C_{i}}]\in H^{1}(Y,$ $O_{C_{i}}(C_{i}))$ for
every $i$, i.e. $\mathfrak{\eta\in}\Xi_{X}^{A_{n}}$.

\begin{proof}
Restricting $\mathfrak{L}_{0}^{A_{n}}$ to $C_{i}$, the corresponding line
bundle summands are
\[
O_{C_{i}}(l_{k})\cong\left\{
\begin{tabular}
[c]{cc}%
$O_{\mathbb{P}^{1}}(1),$ & $k=n+2-i$\\
$O_{\mathbb{P}^{1}}(-1),$ & $k=n+1-i$\\
$O_{\mathbb{P}^{1}},$ & $\text{otherwise.}$%
\end{tabular}
\ \ \ \right.
\]
By Lemma \ref{indep3} and our assumption, we have the proposition.
\end{proof}
\end{proposition}

\subsection{$A_{n}$ Lie algebra bundle $\mathcal{\zeta}_{\mathfrak{\eta}%
}^{A_{n}}$}

As $A_{n}=sl(n+1,\mathbb{C)=}aut_{0}(\mathbb{C}^{n+1})$, $\mathcal{\zeta
}_{\mathfrak{\eta}}^{A_{n}}:=aut_{0}(\mathfrak{L}_{\mathfrak{\eta}}^{A_{n}})$
($\mathfrak{\eta\in}\Xi_{X}^{A_{n}}$) is an $A_{n}$ Lie algebra bundle over
$Y$ which can descend to $X$. This $\mathcal{\zeta}_{\mathfrak{\eta}}^{A_{n}}$
does not depend on the existence of $C_{0}$. And $\mathfrak{L}_{\mathfrak{\eta
}}^{A_{n}}$ is automatically a representation bundle of $\mathcal{\zeta
}_{\mathfrak{\eta}}^{A_{n}}$.

\subsection{$A_{n}$ minuscule representation bundle $\mathfrak{L}%
_{\mathfrak{\eta}}^{(A_{n},{\wedge}^{k}{\mathbb{C}}^{n+1})}$}

Consider a surface $X$ with an $A_{n}$ singularity $p$ and a $(-1)$-curve
$C_{0}$ passing through $p$ with multiplicity $C_{k}$. By Proposition
\ref{card}, $I^{(A_{n},{\wedge}^{k}{\mathbb{C}}^{n+1})}$ has cardinality
$\binom{k}{n+1}$. Define $\mathfrak{L}_{0}^{(A_{n},{\wedge}^{k}{\mathbb{C}%
}^{n+1})}:=\bigoplus_{l\in I}O(l)$ over $Y$.

\begin{lemma}
\label{branch1}$\mathfrak{L}_{0}^{(A_{n},{\wedge}^{k}{\mathbb{C}}^{n+1}%
)}=({\wedge}^{k}\mathfrak{L}_{0}^{A_{n}})(C_{0}^{k}-kC_{0}^{1}-\sum
_{j=1}^{k-1}(k-j)C_{j})$.

\begin{proof}
The bundles on both sides have the same rank, so we only need to check that
every line bundle summand in the right-hand side is $O_{Y}(l)$ for $l$ a
$(-1)$-curve in $I^{(A_{n},{\wedge}^{k}{\mathbb{C}}^{n+1})}$. For any $k$
distinct elements $l_{i_{j}}$ in $I^{(A_{n},{\mathbb{C}}^{n+1})}$, we denote
$l=l_{i_{1}}+l_{i_{2}}+\cdots+l_{i_{k}}+C_{0}^{k}-(l_{1}+l_{2}+\cdots l_{k})$,
then $O_{Y}(l)$ is a summand in the right-hand side. Since the intersection
number of any two distinct $(-1)$-curves in $I^{(A_{n},{\mathbb{C}}^{n+1})}$
is zero, we have $l^{2}=l\cdot K_{Y}=-1$. i.e. $l\in I^{(A_{n},{\wedge}%
^{k}{\mathbb{C}}^{n+1})}$.
\end{proof}
\end{lemma}

From the above lemma and direct computations, for any $C_{i}$,
\[
\mathfrak{L}_{0}^{(A_{n},{\wedge}^{k}{\mathbb{C}}^{n+1})}|_{C_{i}}\cong
O_{\mathbb{P}^{1}}^{\oplus(\binom{k}{n-1}+\binom{k-2}{n-1})}\oplus
(O_{\mathbb{P}^{1}}(1)\oplus O_{\mathbb{P}^{1}}(-1))^{\oplus\binom{k-1}{n-1}%
}\text{.}%
\]

\begin{proposition}
\label{Ak}Fix any $\mathfrak{\eta\in}\Xi_{Y}^{A_{n}}$, there exists a unique
holomorphic structure on $\mathfrak{L}_{0}^{(A_{n},{\wedge}^{k}{\mathbb{C}%
}^{n+1})}$ such that the action of $\mathcal{\zeta}_{\mathfrak{\eta}}^{A_{n}}$
on the resulting bundle $\mathfrak{L}_{\mathfrak{\eta}}^{(A_{n},{\wedge}%
^{k}{\mathbb{C}}^{n+1})}$ is holomorphic. Furthermore, if $\mathfrak{\eta\in
}\Xi_{X}^{A_{n}}$, then $\mathfrak{L}_{\mathfrak{\eta}}^{(A_{n},{\wedge}%
^{k}{\mathbb{C}}^{n+1})}$ can descend to $X$.

\begin{proof}
As the action of $\mathcal{\zeta}_{\mathfrak{\eta}}^{A_{n}}$ on $\mathfrak{L}%
_{\mathfrak{\eta}}^{A_{n}}$ is holomorphic, $\mathcal{\zeta}_{\mathfrak{\eta}%
}^{A_{n}}$ acts on $\mathfrak{L}_{\mathfrak{\eta}}^{(A_{n},{\wedge}%
^{k}{\mathbb{C}}^{n+1})}:=({\wedge}^{k}\mathfrak{L}_{\mathfrak{\eta}}^{A_{n}%
})(C_{0}^{k}-kC_{0}^{1}-\sum_{j=1}^{k-1}(k-j)C_{j})$ holomorphically. The last
assertion follows from Proposition \ref{A1} and the fact that $O(C_{0}%
^{k}-kC_{0}^{1}-\sum_{j=1}^{k-1}(k-j)C_{j})|_{C_{i}}$ is trivial for every
$C_{i}$.
\end{proof}
\end{proposition}

\section{$D_{n}$ case}

We recall that $D_{n}=o(2n,\mathbb{C)=}aut(\mathbb{C}^{2n},q)$ for a
non-degenerate quadratic form $q$ on the standard representation
$\mathbb{C}^{2n}$. The other minuscule representations are ${\mathcal{S}}^{+}$
and ${\mathcal{S}}^{-}$ and the adjoint representation is $\wedge
^{2}\mathbb{C}^{2n}$.

\subsection{$D_{n}$ standard representation bundle $\mathfrak{L}%
_{\mathfrak{\eta}}^{(D_{n},\mathbb{C}^{2n})}$}

We consider a surface $X$ with a $D_{n}$ singularity $p$ and a $(-1)$-curve
$C_{0}$ passing through $p$ with multiplicity $C_{1}$, then $I^{(D_{n}%
,\mathbb{C}^{2n})}=I_{1}\cup I_{2}$ with $I_{1}=\{C_{0}^{1}+\sum_{i=1}%
^{k}C_{i}|0\leq k\leq n-1\}$ and $I_{2}=\{F-l|l\in I_{1}\}$, where
$F=2C_{0}^{1}+2C_{1}+\cdots+2C_{n-2}+C_{n-1}+C_{n}$. We order these
$(-1)$-curves: $l_{k}=F-C_{0}^{1}-\sum_{i=1}^{k-1}C_{i}$ and $l_{2n-k+1}%
=C_{0}^{1}+\sum_{i=1}^{k-1}C_{i}$ for $1\leq k\leq n$.

For any $l_{i}\neq l_{j}\in I$, we have $l_{i}\cdot l_{j}=0$ or $1$. Given any
$l_{i}\in I$, there exists a unique $l_{j}\in I$ such that $l_{i}\cdot
l_{j}=1$. In this case, $l_{i}+l_{j}=F$.

Define $\mathfrak{L}_{0}^{(D_{n},\mathbb{C}^{2n})}:=\bigoplus_{l\in I}O(l)$
over $Y$, for simplicity, we write it as $\mathfrak{L}_{0}^{D_{n}}$. If we
ignore $C_{n}$, then we recover the $A_{n-1}$ case as in the last section.
They are related by the following.

\begin{lemma}
\label{branch2}$\mathfrak{L}_{0}^{D_{n}}=\mathfrak{L}_{0}^{A_{n-1}}%
\oplus(\mathfrak{L}_{0}^{A_{n-1}})^{\ast}(F)$.

\begin{proof}
Since $A_{n-1}$ is a Lie subalgebra of $D_{n}$, we can decompose the
representation of $D_{n}$ as sum of irreducible representations of $A_{n-1}$.
By the branching rule, we have $2n=n+n$, that is ${\mathbb{C}}^{2n}%
={\mathbb{C}}^{n}\oplus({\mathbb{C}}^{n})^{\ast}$ with ${\mathbb{C}}^{2n}$ and
${\mathbb{C}}^{n}$ the standard representations of $D_{n}$ and $A_{n-1}$
respectively. For $I^{(D_{n},\mathbb{C}^{2n})}=I_{1}\cup I_{2}$, $I_{1}$ forms
the standard representation ${\mathbb{C}}^{n}$ of $A_{n-1}$, and $I_{2}$ forms
the $({\mathbb{C}}^{n})^{\ast}$.
\end{proof}
\end{lemma}

From the above lemma and direct computations, for any $C_{i}$,%
\[
\mathfrak{L}_{0}^{D_{n}}|_{C_{i}}\cong O_{\mathbb{P}^{1}}^{\oplus(2n-4)}%
\oplus(O_{\mathbb{P}^{1}}(1)\oplus O_{\mathbb{P}^{1}}(-1))^{\oplus2}\text{.}%
\]

Similar to $(A_{n},$ $\mathbb{C}^{n+1})$ case, we define $\overline{\partial
}_{\mathfrak{\eta}}:\Omega^{0,0}(Y,\mathfrak{L}_{0}^{D_{n}})\longrightarrow
\Omega^{0,1}(Y,\mathfrak{L}_{0}^{D_{n}})$ on $\mathfrak{L}_{0}^{D_{n}%
}=\bigoplus_{k=1}^{2n}O(l_{k})$ by $\overline{\partial}_{\eta}:=\overline
{\partial}_{0}+(\eta_{i,j})_{2n\times2n}$, where $\eta_{i,j}\in\Omega
^{0,1}(Y,O(l_{i}-l_{j}))$ for any $j>i$, otherwise $\eta_{i,j}=0$.

By Lemma \ref{cohomology2} and arguments similar to the proof of Proposition
\ref{holo2} for the $A_{n}$ case, given any $\eta_{i,i+1}$ with $\overline
{\partial}\eta_{i,i+1}=0$ for every $i$, there exists $\eta_{i,j}\in
\Omega^{0,1}(Y,O(l_{i}-l_{j}))$ for every $j>i$ such that $\overline{\partial
}_{\mathfrak{\mathfrak{\eta}}}^{2}=0$.

From the configuration of these $2n$ $(-1)$-curves, we can define a quadratic
form $q$ on the vector space $V_{0}=\mathbb{C}^{I}={\bigoplus_{l\in I}%
}\mathbb{C}\langle v_{l}\rangle$ spanned by these $(-1)$-curves,%
\[
q:V_{0}\otimes V_{0}\longrightarrow\mathbb{C},\text{ }q(v_{l_{i}},v_{l_{j}%
})=l_{i}\cdot l_{j}\text{.}%
\]
The $D_{n}$ Lie algebra is the space of infinitesimal automorphism of $q$,
i.e. $D_{n}=aut(V_{0},q)$.

Correspondingly, we have a fiberwise quadratic form $q$ on the bundle
$\mathfrak{L}_{\mathfrak{\eta}}^{D_{n}}$:

\begin{center}
$q:\mathfrak{L}_{\mathfrak{\eta}}^{D_{n}}\otimes\mathfrak{L}_{\mathfrak{\eta}%
}^{D_{n}}\longrightarrow O\left(  F\right)  $.
\end{center}

\begin{proposition}
\label{q}There exists $\mathfrak{\eta}$ with $\overline{\partial
}_{\mathfrak{\mathfrak{\eta}}}^{2}=0$ such that $\mathfrak{\overline{\partial
}}_{\mathfrak{\mathfrak{\eta}}}q=0$.

\begin{proof}
$\overline{\partial}_{\mathfrak{\mathfrak{\eta}}}q=0$ if and only if
$q(\overline{\partial}_{\mathfrak{\mathfrak{\eta}}}s_{i},s_{j})+q(s_{i}%
,\overline{\partial}_{\mathfrak{\mathfrak{\eta}}}s_{j})=0$ for any $s_{i}\in
H^{0}(Y,O(l_{i}))$ and $s_{j}\in H^{0}(Y,O(l_{j}))$. From the definition of
$q$, this is equivalent to $\eta_{2n+1-j,i}+\eta_{2n+1-i,j}=0$, i.e.
$\eta_{i,j}=-\eta_{2n+1-j,2n+1-i}$ for any $j>i$. From $l_{i}+l_{2n+1-i}%
=l_{j}+l_{2n+1-j}=F$, we have%
\[
\eta_{i,j}\in\Omega^{0,1}(Y,O(l_{i}-l_{j}))=\Omega^{0,1}(Y,O(l_{2n+1-j}%
-l_{2n+1-i}))\ni\eta_{2n+1-j,2n+1-i}\text{.}%
\]

We construct $\mathfrak{\eta}$ which satisfies $\overline{\partial
}_{\mathfrak{\mathfrak{\eta}}}^{2}=0$ with $\eta_{i,j}=-\eta_{2n+1-j,2n+1-i}$
inductively on $j-i$. For $j-i=1$, we can always take $\eta_{i,i+1}%
=-\eta_{2n-i,2n+1-i}$. Note we have $\eta_{n,n+1}=0$. For $j-i=2$, we have%
\[
\overline{\partial}\eta_{i,i+2}=-\eta_{i,i+1}\eta_{i+1,i+2},
\]%
\[
\overline{\partial}\eta_{2n-i-1,2n-i+1}=-\eta_{2n-i-1,2n-i}\eta_{2n-i,2n-i+1}%
=-\eta_{i+1,i+2}\eta_{i,i+1}=-\overline{\partial}\eta_{i,i+2}\text{,}%
\]
so we can take $\eta_{i,i+2}=-\eta_{2n-i-1,2n-i+1}$.

Repeat this process inductively on $j-i$, we can take $\eta_{i,j}%
=-\eta_{2n+1-j,2n+1-i}$ for any $j>i$. So there exists
$\mathfrak{\mathfrak{\eta}} $ satisfying $\mathfrak{\overline{\partial}%
}_{\mathfrak{\mathfrak{\eta}}}q=0$.
\end{proof}
\end{proposition}

Until now, we have proved $\Xi_{Y}^{D_{n}}$ is not empty.

Restricting $\mathfrak{L}_{0}^{D_{n}}$ to $C_{n}$, the corresponding line
bundle summands are:
\[
O_{C_{n}}(l_{j})\cong\left\{
\begin{tabular}
[c]{cc}%
$O_{\mathbb{P}^{1}}(1),$ & $j=n+1\text{ or }n+2$\\
$O_{\mathbb{P}^{1}}(-1),$ & $j=n-1\text{ or }n$\\
$O_{\mathbb{P}^{1}},$ & $\text{otherwise.}$%
\end{tabular}
\ \ \right.
\]
The pairs of $O_{\mathbb{P}^{1}}(\pm1)$ in $\mathfrak{L}_{0}^{D_{n}}|_{C_{n}}$
are given by $\{l_{n-1},l_{n+1}\}$ and $\{l_{n},l_{n+2}\}$. To construct a
trivialization of $\mathfrak{L}_{\mathfrak{\eta}}^{D_{n}}|_{C_{n}}$, we need
the following generalizations of Lemma \ref{indep} and Lemma \ref{indep3}.

\begin{lemma}
\label{indep2}Under the same assumption as in Lemma \ref{section}. Assume
$l_{i+1},l_{i+2}\cdots l_{i+2k}$ satisfy $l_{i+j}\cdot C=-1$ and
$l_{i+k+j}=l_{i+j}-C$ for $j=1,2,\cdots k$. If $\mathfrak{\eta}_{i+p,i+q}=0$
for $2\leq p\leq k$, $k+1\leq q\leq2k-1$ and $q-p\leq k-1$, i.e. the
corresponding submatrix of $\overline{\partial}_{\mathfrak{L}}$ given by
$l_{i+1},l_{i+2},\cdots l_{i+2k}$ looks like
\[
\left(
\begin{tabular}
[c]{c|c}%
$\ast$ & $%
\begin{array}
[c]{cccc}%
\eta_{i+1,i+k+1} & \eta_{i+1,i+k+2} & \cdots & \eta_{i+1,i+2k}\\
0 & \eta_{i+2,i+k+2} & \cdots & \eta_{i+2,i+2k}\\
\vdots & \vdots & \ddots & \vdots\\
0 & 0 & \cdots & \eta_{i+k,i+2k}%
\end{array}
$\\\hline
$0_{k\times k}$ & $\ast$%
\end{tabular}
\ \right)
\]
with $\eta_{i+1,i+k+1},$ $\eta_{i+2,i+k+2}\cdots\eta_{i+k,i+2k}$ in
$\Omega^{0,1}(Y,O(C))$. Suppose $[\eta_{i+1,i+k+1}|_{C}],$ $[\eta
_{i+2,i+k+2}|_{C}],\cdots,[\eta_{i+k,i+2k}|_{C}]$ are nonzero, we can
construct $2k$ holomorphic sections of $\mathfrak{L}|_{C}$ which are linearly
independent at every point in $C$.

\begin{proof}
In order to keep our notations simpler, we assume $k=2$. The above matrix
given by $l_{i+1},l_{i+2},l_{i+3},l_{i+4}$ has the form
\[
\left(
\begin{tabular}
[c]{c|c}%
$%
\begin{array}
[c]{cc}%
\overline{\partial} & \eta_{i+1,i+2}\\
0 & \overline{\partial}%
\end{array}
$ & $%
\begin{array}
[c]{cc}%
\eta_{i+1,i+3} & \ast\\
0 & \eta_{i+2,i+4}%
\end{array}
$\\\hline
$0_{2\times2}$ & $%
\begin{array}
[c]{cc}%
\text{ }\overline{\partial}\text{ } & \text{ }\eta_{i+3,i+4}\\
0 & \overline{\partial}%
\end{array}
$%
\end{tabular}
\ \ \right)  \text{.}%
\]
From $H^{0}(Y,O_{C}(l_{i+4}))\cong H^{0}(\mathbb{P}^{1},O(1))\cong{\mathbb{C}%
}^{2}$ and $[\eta_{i+2,i+4}|_{C}]\neq0$, there exist two holomorphic sections
of $\mathfrak{L}|_{C}$ which are linearly independent at every point in $C$:
$s_{1}=(y_{1},u_{1},x_{1},t_{1})^{t}$ and $s_{2}=(y_{2},u_{2},x_{2},t_{2}%
)^{t}$ with $u_{1},t_{1},u_{2},t_{2}$ given in Lemma \ref{indep}. Similarly,
from $H^{0}(Y,O_{C}(l_{i+3}))\cong{\mathbb{C}}^{2}$ and $[\eta_{i+1,i+3}%
|_{C}]\neq0$, we also have two holomorphic sections of $\mathfrak{L}|_{C}$
which are linearly independent at every point in $C$: $s_{3}=(y_{3}%
,0,x_{3},0)^{t}$ and $s_{4}=(y_{4},0,x_{4},0)^{t}$. If there exist
$a_{1},a_{2},a_{3},a_{4}$ such that $a_{1}s_{1}+$ $a_{2}s_{2}+a_{3}s_{3}%
+a_{4}s_{4}=0$ at some point in $C$, then we have $a_{1}t_{1}+$ $a_{2}t_{2}=0$
and $a_{1}u_{1}+$ $a_{2}u_{2}=0$ at some point, which is impossible by the
explicit formulas for $u_{1},t_{1},u_{2},t_{2}$ in Lemma \ref{indep}. Hence we
have the lemma.
\end{proof}
\end{lemma}

\begin{lemma}
\label{indep4}Under the same assumption as in Lemma \ref{section}, we assume
$\mathfrak{L}|_{C}\cong O_{\mathbb{P}^{1}}^{\oplus m}\oplus(O_{\mathbb{P}^{1}%
}(1)\oplus O_{\mathbb{P}^{1}}(-1))^{\oplus n}$ with each pair of
$O_{\mathbb{P}^{1}}(\pm1)$ and the corresponding holomorphic structure as in
Lemma \ref{indep2}. Then $\mathfrak{L}|_{C}$ is trivial if and only if
$[\mathfrak{\eta}_{i,j}|_{C}]\neq0$ for any $\mathfrak{\eta}_{i,j}\in
\Omega^{0,1}(Y,O(C))$.

\begin{proof}
Same arguments as in the proof of Lemma \ref{indep2} and Lemma \ref{indep3}.
\end{proof}
\end{lemma}

\begin{proposition}
\label{D1}The bundle $\mathfrak{L}_{\mathfrak{\eta}}^{D_{n}}$ over $Y$ with
$\mathfrak{\eta\in}\Xi_{Y}^{D_{n}}$ can descend to $X$ if and only if for
every $C_{k}$ and $\eta_{i,j}\in\Omega^{0,1}(Y,O(C_{k}))$, $[\eta
_{i,j}|_{C_{k}}]\neq0$, i.e. $\mathfrak{\eta\in}\Xi_{X}^{D_{n}}$.

\begin{proof}
Restricting $\mathfrak{L}_{0}^{D_{n}}$ to $C_{i}$ $(1\leq i\leq n-1)$, the
line bundle summands are
\[
O_{C_{i}}(l_{j})\cong\left\{
\begin{tabular}
[c]{cc}%
$O_{\mathbb{P}^{1}}(1),$ & $j=i+1\text{ or }2n-i$\\
$O_{\mathbb{P}^{1}}(-1),$ & $j=i\text{ or }2n-i+1$\\
$O_{\mathbb{P}^{1}},$ & $\text{otherwise.}$%
\end{tabular}
\ \ \ \right.
\]
By Lemma \ref{indep3}, $\mathfrak{L}_{\mathfrak{\eta}}^{D_{n}}|_{C_{i}}$ is
trivial if and only if $[\eta_{i,i+1}|_{C_{i}}]$, $[\eta_{2n-i,2n+1-i}%
|_{C_{i}}]$ are not zeros. For $C_{n}$, The pairs of $O_{\mathbb{P}^{1}}%
(\pm1)$ in $\mathfrak{L}_{0}^{D_{n}}|_{C_{n}}$ are given by $\{l_{n-1}%
,l_{n+1}\}$ and $\{l_{n},l_{n+2}\}$. By Lemma \ref{indep4} and $\eta
_{n,n+1}=0$ $($Proposition \ref{q}$)$, $\mathfrak{L}_{\mathfrak{\eta}}^{D_{n}%
}|_{C_{n}}$ is trivial if and only if $[\eta_{n-1,n+1}|_{C_{n}}]$,
$[\eta_{n,n+2}|_{C_{n}}]$ are not zeros. In fact, this $\mathfrak{L}%
_{\mathfrak{\eta}}^{D_{n}}$ is just an extension of $\mathfrak{L}%
_{\mathfrak{\eta}^{\prime}}^{A_{n-1}}$ by $(\mathfrak{L}_{\mathfrak{\eta
}^{\prime}}^{A_{n-1}})^{\ast}(F)$ for some $\mathfrak{\eta}^{\prime
}\mathfrak{\in}\Xi_{X}^{A_{n-1}}$ with $\mathfrak{\eta}^{\prime}%
\subset\mathfrak{\eta}$.
\end{proof}
\end{proposition}

\subsection{$D_{n}$ Lie algebra bundle $\mathcal{\zeta}_{\mathfrak{\eta}%
}^{D_{n}}$}

Note that $\mathcal{\zeta}_{\mathfrak{\eta}}^{D_{n}}=aut(\mathfrak{L}%
_{\mathfrak{\eta}}^{D_{n}},q)$ is a $D_{n}$ Lie algebra bundle over $Y$. In
order for $\mathcal{\zeta}_{\mathfrak{\eta}}^{D_{n}}$ to descend to $X$ as a
Lie algebra bundle, we need to show that $q|_{C_{i}}:\mathfrak{L}%
_{\mathfrak{\eta}}^{D_{n}}|_{C_{i}}\otimes\mathfrak{L}_{\mathfrak{\eta}%
}^{D_{n}}|_{C_{i}}\longrightarrow O_{C_{i}}\left(  F\right)  $ is a constant
map for every $C_{i}$. This follows from the fact that both $\mathfrak{L}%
_{\mathfrak{\eta}}^{D_{n}}$ and $O(F)$ are trivial on all $C_{i}$'s and
$\overline{\partial}_{\mathfrak{\eta}}q=0$. From the construction,
$\mathfrak{L}_{\mathfrak{\eta}}^{D_{n}}$ is a representation bundle of
$\mathcal{\zeta}_{\mathfrak{\eta}}^{D_{n}}$.

\subsection{$D_{n}$ spinor representation bundles $\mathfrak{L}%
_{\mathfrak{\eta}}^{(D_{n},{\mathcal{S}}^{\pm})}$}

We will only deal with ${\mathcal{S}}^{+}$, as ${\mathcal{S}}^{-}$ case is
analogous. Consider a surface $X$ with a $D_{n}$ singularity $p$ and a
$(-1)$-curve $C_{0}$ passing through $p$ with multiplicity $C_{n}$. By
Proposition \ref{card}, $|I^{(D_{n},{\mathcal{S}}^{+})}|=2^{n-1}$. Define
$\mathfrak{L}_{0}^{(D_{n},{\mathcal{S}}^{+})}:=\bigoplus_{l\in I}O(l)$ over
$Y$.

\begin{lemma}
\label{branch3}$\mathfrak{L}_{0}^{(D_{n},{\mathcal{S}}^{+})}=\bigoplus
_{m=0}^{[\frac{n}{2}]}{\wedge}^{2m}(\mathfrak{L}_{0}^{A_{n-1}})^{\ast
}(mF+C_{0}^{n})$.

\begin{proof}
First we check that every line bundle summand in the right-hand side is
$O_{Y}(l)$ for a $(-1)$-curve $l$ in $I^{(D_{n},{\mathcal{S}}^{+})}$. For any
$l_{i}\in I^{(A_{n-1},\mathbb{C}^{n})}$, we have $l_{i}$ $\cdot$ $C_{0}^{n}%
=0$, $l_{i}$ $\cdot$ $F=0$ and $F$ $\cdot$ $F=0$, $F$ $\cdot$ $C_{0}^{n}=1$.
For any $2m$ distinct elements $l_{i_{j}}$'s in $I^{(A_{n-1},\mathbb{C}^{n})}%
$, we denote $l=-(l_{i_{1}}+\cdots l_{i_{2m}})+mF+C_{0}^{n}$, then $O_{Y}(l)$
is a summand in the right-hand side. Since $l^{2}=-1$ and $l\cdot K_{Y}=-1$,
$l\in I^{(D_{n},{\mathcal{S}}^{+})}$. Also the rank of these two bundles are
the same which is $2^{n-1}=\binom{n}{0}+\binom{n}{2}+\cdots+\binom{n}%
{2[\frac{n}{2}]}$. Hence we have the lemma.
\end{proof}
\end{lemma}

From the above lemma and direct computations, for any $C_{i}$,%
\[
\mathfrak{L}_{0}^{(D_{n},{\mathcal{S}}^{+})}|_{C_{i}}\cong O_{\mathbb{P}^{1}%
}^{\oplus2^{n-2}}\oplus(O_{\mathbb{P}^{1}}(1)\oplus O_{\mathbb{P}^{1}%
}(-1))^{\oplus2^{n-3}}\text{.}%
\]

The $D_{n}$ Lie algebra bundle $\mathcal{\zeta}_{0}^{D_{n}}$ has a natural
fiberwise action on $\mathfrak{L}_{0}^{(D_{n},{\mathcal{S}}^{+})}$,%
\[
\rho:\mathcal{\zeta}_{0}^{D_{n}}\otimes\mathfrak{L}_{0}^{(D_{n},{\mathcal{S}%
}^{+})}\longrightarrow\mathfrak{L}_{0}^{(D_{n},{\mathcal{S}}^{+})}\text{,}%
\]
which can be described easily using the reduction to $A_{n-1}$ (with the node
$C_{n}$ being removed): recall%
\[
\mathcal{\zeta}_{0}^{D_{n}}=({\wedge}^{2}\mathfrak{L}_{0}^{A_{n-1}}%
(-F))\oplus((\mathfrak{L}_{0}^{A_{n-1}})^{\ast}\otimes\mathfrak{L}%
_{0}^{A_{n-1}})\oplus(({\wedge}^{2}\mathfrak{L}_{0}^{A_{n-1}})^{\ast
}(F))\text{,}%
\]%
\[
\mathfrak{L}_{0}^{(D_{n},{\mathcal{S}}^{+})}=\bigoplus_{m=0}^{[\frac{n}{2}%
]}{\wedge}^{2m}(\mathfrak{L}_{0}^{A_{n-1}})^{\ast}(mF)\text{,}\footnote{For
simplicity, we omit the $C_{0}^{n}$ factor.}%
\]
and $\rho$ is given by interior and exterior multiplications for ${\wedge
}^{\cdot}\mathfrak{L}_{0}^{A_{n-1}}$.

\begin{proposition}
\label{action}Fix any $\mathfrak{\eta\in}\Xi_{Y}^{D_{n}}$, there exists a
unique holomorphic structure on $\mathfrak{L}_{0}^{(D_{n},{\mathcal{S}}^{+})}$
such that the action of $\mathcal{\zeta}_{\mathfrak{\eta}}^{D_{n}}$ on the
resulting bundle $\mathfrak{L}_{\mathfrak{\eta}}^{(D_{n},{\mathcal{S}}^{+})}$
is holomorphic. Furthermore, if $\mathfrak{\eta\in}\Xi_{X}^{D_{n}}$, then
$\mathfrak{L}_{\mathfrak{\eta}}^{(D_{n},{\mathcal{S}}^{+})}$ can descend to
$X$.

\begin{proof}
First, we recall the holomorphic structure on $\mathcal{\zeta}_{\mathfrak{\eta
}}^{D_{n}}$. In $I^{(D_{n},\mathbb{C}^{2n})}=I_{1}\cup I_{2}$ with
$I_{1}=\{l_{i}=C_{0}^{1}+\sum_{m=1}^{2n-i}C_{m}|n+1\leq i\leq2n\}$ and
$I_{2}=\{F-l_{i}|l_{i}\in I_{1}\}$, let $s_{i},$ $s_{i}^{\ast}$ and $f$ be
local holomorphic sections of $O(l_{i}),O(F-l_{i})$ and $O(-F)$ respectively.
By Proposition \ref{q}, we have%
\[
\overline{\partial}_{\mathfrak{L}_{\mathfrak{\eta}}^{D_{n}}}s_{i}^{\ast}%
=\sum_{p=1}^{i-1}\eta_{p,i}s_{p}^{\ast}%
\]
and%
\[
\overline{\partial}_{\mathfrak{L}_{\mathfrak{\eta}}^{D_{n}}}s_{i}=\sum
_{p=1}^{n}\eta_{p,2n+1-i}s_{p}^{\ast}-\sum_{p=i+1}^{n}\eta_{i,p}s_{p}\text{.}%
\]
Back to $\left(  \mathfrak{L}_{0}^{{\mathcal{S}}^{+}}\right)  ^{\ast}$, we
define $s_{i_{1}\cdots i_{2m}}:=s_{i_{1}}\wedge\cdots\wedge s_{i_{2m}}\otimes
f^{m}\in\Gamma({\wedge}^{2m}\mathfrak{L}_{0}^{A_{n-1}}(-mF))$ where $i_{j}%
\in\{1,2,\cdots n\}$ and define $\overline{\partial}_{\left(  \mathfrak{L}%
_{\mathfrak{\eta}}^{(D_{n},{\mathcal{S}}^{+})}\right)  ^{\ast}}$ as follows:
\[
\overline{\partial}_{\mathfrak{L}}s_{i_{1}\cdots i_{2m}}=\sum_{p,q}%
(-1)^{p+q}\eta_{i_{p},2n+1-i_{q}}s_{i_{1}\cdots\widehat{i_{p}}\cdots
\widehat{i_{q}}\cdots i_{2m}}-\sum_{p}\sum_{k\neq i_{p}}\eta_{i_{p},k}%
s_{i_{1}\cdots i_{p-1}ki_{p+1}\cdots i_{2m}}\text{,}%
\]
where $\widehat{i_{j}}$ means deleting the $i_{j}$ component. We verify
$\overline{\partial}_{\mathfrak{L}}^{2}=0$ by direct computations.

We claim that $\overline{\partial}_{\mathfrak{L}}$ is the unique holomorphic
structure such that the action of $\mathcal{\zeta}_{\mathfrak{\eta}}^{D_{n}}$
on $\left(  \mathfrak{L}_{\mathfrak{\eta}}^{(D_{n},{\mathcal{S}}^{+})}\right)
^{\ast}$ is holomorphic, i.e.%
\begin{equation}
\overline{\partial}_{\mathcal{\zeta}_{\mathfrak{\eta}}^{D_{n}}}(g)\cdot
x+g\cdot(\overline{\partial}_{\mathfrak{L}}x)=\overline{\partial
}_{\mathfrak{L}}(g\cdot x)\text{ } \tag{$\ast$}%
\end{equation}
for any $g\in\Gamma(\mathcal{\zeta}_{\mathfrak{\eta}}^{D_{n}})$ and
$x\in\Gamma\left(  \left(  \mathfrak{L}_{0}^{{\mathcal{S}}^{+}}\right)
^{\ast}\right)  $.

We prove the above claim by induction on $m$. When $m=0,$ $x=s_{0}\in
\Gamma(\wedge^{0}\mathfrak{L}_{0}^{A_{n-1}})$, by direct computations,
$(\ast)$ holds for any $g\in\Gamma(\mathcal{\zeta}_{\mathfrak{\eta}}^{D_{n}})$
if and only if $\overline{\partial}_{\mathfrak{L}}s_{0}=0$ and $\overline
{\partial}_{\mathfrak{L}}s_{ij}=-\eta_{i,2n+1-j}s_{0}-\sum_{p=i+1}^{n}%
\eta_{i,p}s_{pj}-\sum_{p=j+1}^{n}\eta_{j,p}s_{ip}$ for any $s_{ij}\in
\Gamma({\wedge}^{2}\mathfrak{L}_{0}^{A_{n-1}})$. When $m=2$, from the above
formula for $\overline{\partial}_{\mathfrak{L}}s_{ij}$, we can get the formula
for $\overline{\partial}_{\mathfrak{L}}s_{ijkl}$. Repeat this process
inductively, we can get the above formula for $\overline{\partial
}_{\mathfrak{L}}s_{i_{1}\cdots i_{2m}}$. Hence we have the first part of this proposition.

For the second part, we will rewrite $\overline{\partial}_{\mathfrak{L}}$ in
matrix form. Firstly, we have%
\[
\overline{\partial}_{\mathfrak{L}_{\mathfrak{\eta}}^{(D_{n},\mathbb{C}^{2n})}%
}=\left(
\begin{tabular}
[c]{c|c}%
$\overline{\partial}_{(\mathfrak{L}_{\mathfrak{\eta}^{\prime}}^{A_{n-1}%
})^{\ast}(F)}$ & $B$\\\hline
$0$ & $\overline{\partial}_{\mathfrak{L}_{\mathfrak{\eta}^{\prime}}^{A_{n-1}}%
}$%
\end{tabular}
\ \ \ \ \right)
\]
with $\mathfrak{\eta}^{\prime}\subset\mathfrak{\eta}$ and the upper right
block $B$ has the following shape%
\[
B=\left(
\begin{tabular}
[c]{cc}%
$%
\begin{array}
[c]{cc}%
\vdots & \vdots
\end{array}
$ & $\ddots$\\
$%
\begin{array}
[c]{cc}%
\beta & \ast\\
0 & -\beta
\end{array}
$ & $%
\begin{array}
[c]{c}%
\cdots\\
\cdots
\end{array}
$%
\end{tabular}
\ \ \ \ \right)  \text{,}%
\]
for $[\mathfrak{\beta]\in}H^{1}(Y,O(C_{n}))$.

In particular, we have an exact sequence of holomorphic bundles:%
\begin{equation}
0\rightarrow(\mathfrak{L}_{\mathfrak{\eta}^{\prime}}^{A_{n-1}})^{\ast
}(F)\rightarrow\mathfrak{L}_{\mathfrak{\eta}}^{D_{n}}\rightarrow
\mathfrak{L}_{\mathfrak{\eta}^{\prime}}^{A_{n-1}}\rightarrow0\text{.} \tag{
$\Delta$}%
\end{equation}

By tensoring $(\Delta)$ with $\mathfrak{L}_{\mathfrak{\eta}^{\prime}}%
^{A_{n-1}}(-F)$, we obtain a bundle $S_{1}$ as follows,%
\[
0\rightarrow O_{Y}\rightarrow S_{1}\rightarrow{\wedge}^{2}\mathfrak{L}%
_{\mathfrak{\eta}^{\prime}}^{A_{n-1}}(-F)\rightarrow0\text{,}%
\]
with the induced holomorphic structure given by%
\[
\overline{\partial}_{S_{1}}=\left(
\begin{tabular}
[c]{c|c}%
$\overline{\partial}_{{\wedge}^{0}\mathfrak{L}_{\mathfrak{\eta}^{\prime}%
}^{A_{n-1}}}$ & $B_{1}$\\\hline
$0$ & $\overline{\partial}_{{\wedge}^{2}\mathfrak{L}_{\mathfrak{\eta}^{\prime
}}^{A_{n-1}}(-F)}$%
\end{tabular}
\ \right)  =\left(
\begin{tabular}
[c]{c|c}%
$\overline{\partial}_{{\wedge}^{0}\mathfrak{L}_{\mathfrak{\eta}^{\prime}%
}^{A_{n-1}}}$ & $%
\begin{array}
[c]{cc}%
\pm\beta & \cdots
\end{array}
$\\\hline
$0$ & $\overline{\partial}_{{\wedge}^{2}\mathfrak{L}_{\mathfrak{\eta}^{\prime
}}^{A_{n-1}}(-F)}$%
\end{tabular}
\ \right)  \text{.}%
\]
The occurrence of $\pm\beta$ in that location is because $l_{n+1}+l_{n+2}$
with $l_{n+1},l_{n+2}\in I^{(D_{n},\mathbb{C}^{2n})}$ is the largest element
in $I^{(A_{n-1},{\wedge}^{2}\mathbb{C}^{n})}$ and $F-l_{n+1}-l_{n+2}=C_{n}$
because $F=2C_{0}^{1}+2C_{1}+\cdots2C_{n-2}+C_{n-1}+C_{n}$.

Similarly, we have an extension bundle%
\[
0\rightarrow{\wedge}^{2}\mathfrak{L}_{\mathfrak{\eta}^{\prime}}^{A_{n-1}%
}(-F)\rightarrow S_{2}\rightarrow{\wedge}^{4}\mathfrak{L}_{\mathfrak{\eta
}^{\prime}}^{A_{n-1}}(-2F)\rightarrow0\text{,}%
\]
with%
\[
\overline{\partial}_{S_{2}}=\left(
\begin{tabular}
[c]{c|c}%
$\overline{\partial}_{{\wedge}^{2}\mathfrak{L}_{\mathfrak{\eta}^{\prime}%
}^{A_{n-1}}(-F)}$ & $B_{2}$\\\hline
$0$ & $\overline{\partial}_{{\wedge}^{4}\mathfrak{L}_{\mathfrak{\eta}^{\prime
}}^{A_{n-1}}(-2F)}$%
\end{tabular}
\ \ \ \right)  \text{,}%
\]
where%
\[
B_{2}=\left(
\begin{array}
[c]{ccccc}
&  &  &  & \\
\pm\beta &  &  &  & \\
0 & \pm\beta &  &  & \\
\vdots & \vdots & \ddots &  & \\
0 & 0 & \cdots & \pm\beta &
\end{array}
\right)  \text{,}%
\]
for $[\mathfrak{\beta]\in}H^{1}(Y,Hom(O(l_{i}+l_{j}+l_{n+1}+l_{n+2}%
-2F),O(l_{i}+l_{j}-F)))=H^{1}(Y,O(C_{n}))$ with $i,j\in\{n+3,n+4,\cdots,2n\}$.
And the number of $\pm\beta$'s is $\binom{n-2}{2}$.

Inductively, we obtain $\overline{\partial}_{\left(  \mathfrak{L}%
_{\mathfrak{\eta}}^{(D_{n},{\mathcal{S}}^{+})}\right)  ^{\ast}}$ as above
which has the shape that satisfies Lemma \ref{indep4}:%
\[
\overline{\partial}_{\left(  \mathfrak{L}_{\mathfrak{\eta}}^{(D_{n}%
,{\mathcal{S}}^{+})}\right)  ^{\ast}}=\left(
\begin{tabular}
[c]{c|c|c|c}%
$\overline{\partial}_{\wedge^{0}\mathfrak{L}_{\mathfrak{\eta}^{\prime}%
}^{A_{n-1}}}$ & $B_{1}$ & $\cdots$ & $\cdots$\\\hline
$0$ & $\overline{\partial}_{{\wedge}^{2}\mathfrak{L}_{\mathfrak{\eta}^{\prime
}}^{A_{n-1}}(-F)}$ & $B_{2}$ & $\ddots$\\\hline
$0$ & $0$ & $\overline{\partial}_{\wedge^{4}\mathfrak{L}_{\mathfrak{\eta
}^{\prime}}^{A_{n-1}}(-2F)}$ & $\ddots$\\\hline
$\vdots$ & $\ddots$ & $\ddots$ & $\ddots$%
\end{tabular}
\ \ \ \right)  \text{.}%
\]
The number of $\pm\mathfrak{\beta\in}\Omega^{1}(Y,O(C_{n}))$ in $\overline
{\partial}_{\mathfrak{L}}$ is $\binom{n-2}{0}+\binom{n-2}{2}+\cdots
+\binom{n-2}{2[\frac{n-2}{2}]}=2^{n-3}$.

To prove that $\left(  \mathfrak{L}_{\mathfrak{\eta}}^{{\mathcal{S}}^{+}%
}\right)  ^{\ast}$ can descend to $X$ when $\mathfrak{\eta\in}\Xi_{X}^{D_{n}}%
$, we need to show $\left(  \mathfrak{L}_{\mathfrak{\eta}}^{{\mathcal{S}}^{+}%
}\right)  ^{\ast}|_{C_{i}}$ is trivial for every $C_{i}$. When $i\neq n$, this
follows from the fact that $\mathfrak{L}_{\mathfrak{\eta}^{\prime}}^{A_{n-1}}$
is trivial $($Proposition \ref{A1}$)$ and $Ext_{\mathbb{P}^{1}}^{1}%
(O,O)\cong0$. When $i=n$, this follows from Lemma \ref{indep4} and
$\mathfrak{\beta=\eta_{n-1,n+1}\in}\Omega^{0,1}(Y,O(C_{n}))$ with
$[\eta_{n-1,n+1}|_{C_{n}}]\neq0$.
\end{proof}
\end{proposition}

\section{$E_{n}$ case}

\subsection{$E_{6}$ case}

We recall that \cite{A} $E_{6}\mathbb{=}aut(\mathbb{C}^{27},c)$ for a
non-degenerate cubic form $c$ on the standard representation $\mathbb{C}^{27}%
$. The other minuscule representation is $\overline{{\mathbb{C}}^{27}}$.

We consider a surface $X$ with an $E_{6}$ singularity $p$ and a $(-1)$-curve
$C_{0}$ passing through $p$ with multiplicity $C_{1}$. By Proposition
\ref{card}, $I^{(E_{6},\mathbb{C}^{27})}$ has cardinality $27$. For any two
distinct $(-1)$-curves $l_{i}$ and $l_{j}$ in $I$, we have $l_{i}\cdot
l_{j}=0$ or $1$.

Define $\mathfrak{L}_{0}^{(E_{6},\mathbb{C}^{27})}:=\bigoplus_{l\in I}O(l)$
over $Y$, for simplicity, we write it as $\mathfrak{L}_{0}^{E_{6}}$. If we
ignore $C_{6}$, then we recover the $A_{5}$ case as in $\S 4.1$.

\begin{lemma}
\label{branch4}$\mathfrak{L}_{0}^{E_{6}}=\mathfrak{L}_{0}^{A_{5}}%
\oplus({\wedge}^{2}\mathfrak{L}_{0}^{A_{5}})^{\ast}(H)\oplus({\wedge}%
^{5}\mathfrak{L}_{0}^{A_{5}})^{\ast}(2H)$, where $H=3C_{0}^{1}+3C_{1}%
+3C_{2}+3C_{3}+2C_{4}+C_{5}+C_{6}$.

\begin{proof}
$E_{6}$ has $A_{5}$ as a Lie subalgebra, the branching rule is $27=6+15+6$,
i.e. ${\mathbb{C}}^{27}={\mathbb{C}}^{6}\oplus{\wedge}^{2}({\mathbb{C}}%
^{6})^{\ast}\oplus{\wedge}^{5}({\mathbb{C}}^{6})^{\ast}$. The first $6$
$(-1)$-curves in $I:$ $l_{1}=C_{0}^{1},$ $l_{2}=C_{0}^{1}+C_{1},\cdots$
$l_{6}=C_{0}^{1}+C_{1}+C_{2}+C_{3}+C_{4}+C_{5}$ form the standard
representation ${\mathbb{C}}^{6}$ of $A_{5}$. The next $15$ $(-1)$-curves are
given by $H-l_{i}-l_{j}$ with $i\neq j\in\{1,2,\cdots,6\}$. The remaining $6$
$(-1)$-curves are given by $2H-l_{1}-l_{2}-\cdots-\widehat{l_{i}}-\cdots
-l_{6}$.
\end{proof}
\end{lemma}

From the above lemma and direct computations, for any $C_{i}$,%
\[
\mathfrak{L}_{0}^{E_{6}}|_{C_{i}}\cong O_{\mathbb{P}^{1}}^{\oplus15}%
\oplus(O_{\mathbb{P}^{1}}(1)\oplus O_{\mathbb{P}^{1}}(-1))^{\oplus6}\text{.}%
\]

From Lemma \ref{branch4}, we can easily determine the configuration of these
$27$ $(-1)$-curves \cite{M}: Fix any $(-1)$-curve, there are exactly $10$
$(-1)$-curves intersect it, together with the fixed $(-1)$-curve, they form
$5$ triangles. A triple $l_{i},l_{j},l_{k}$ is called a triangle if
$l_{i}+l_{j}+l_{k}=K^{\prime}$, where $K^{\prime}=3C_{0}^{1}+4C_{1}%
+5C_{2}+6C_{3}+4C_{4}+2C_{5}+3C_{6}$.

From the configuration of these $27$ $(-1)$-curves in $Y$, we can define a
cubic form $c$ on the vector space $V_{0}=\mathbb{C}^{I}={\bigoplus_{l\in I}%
}\mathbb{C}\langle v_{l}\rangle$ spanned by $(-1)$-curves,%
\[
c:V_{0}\otimes V_{0}\otimes V_{0}\longrightarrow\mathbb{C},\text{ }(v_{l_{i}%
},v_{l_{j}},v_{l_{k}})\mapsto\left\{
\begin{tabular}
[c]{ll}%
$\pm1$ & if $l_{i}+l_{j}+l_{k}=K^{\prime}$\\
$0$ & otherwise.
\end{tabular}
\ \ \ \ \ \ \right.
\]
The signs above can be determined explicitly \cite{A}\cite{HT} such that
$E_{6}=aut(V_{0},c)$.

Correspondingly, we have a fiberwise cubic form $c$ on the bundle
$\mathfrak{L}_{\mathfrak{\eta}}^{E_{6}}$,

\begin{center}
$c:\mathfrak{L}_{\mathfrak{\eta}}^{E_{6}}\otimes\mathfrak{L}_{\mathfrak{\eta}%
}^{E_{6}}\otimes\mathfrak{L}_{\mathfrak{\eta}}^{E_{6}}\longrightarrow O\left(
K^{\prime}\right)  $.
\end{center}

\begin{proposition}
\label{c}There exists $\mathfrak{\eta}$ with $\overline{\partial
}_{\mathfrak{\eta}}^{2}=0$ such that $\mathfrak{\overline{\partial}%
}_{\mathfrak{\mathfrak{\eta}}}c=0$.

\begin{proof}
Note $\overline{\partial}_{\mathfrak{\mathfrak{\eta}}}c=0$ if and only if%
\begin{equation}
c(\overline{\partial}_{\mathfrak{\mathfrak{\eta}}}s_{i},s_{j},s_{k}%
)+c(s_{i},\overline{\partial}_{\mathfrak{\mathfrak{\eta}}}s_{j},s_{k}%
)+c(s_{i},s_{j},\overline{\partial}_{\mathfrak{\mathfrak{\eta}}}s_{k})=0\text{
} \tag{$\ast$}%
\end{equation}
for any $s_{i}\in H^{0}(Y,O(l_{i}))$, $s_{j}\in H^{0}(Y,O(l_{j}))$ and
$s_{k}\in H^{0}(Y,O(l_{k}))$. From the definition of $c$, if $l_{i}%
+l_{j}+l_{k}=K^{\prime}$, then the above equation $(\ast)$ holds
automatically. If $l_{i}+l_{j}+l_{k}\neq K^{\prime}$, without loss of
generality, we assume $l_{i}\cdot l_{j}=0$, then we have the following four cases.

Case $(i)$, if $l_{i}\cdot l_{k}=0$ and $l_{j}\cdot l_{k}=0$, then $(\ast)$
holds automatically.

Case $(ii)$, if $l_{i}\cdot l_{k}=0$ and $l_{j}\cdot l_{k}=1$, then $(\ast)$
holds if $\mathfrak{\eta}_{l_{i},K^{\prime}-l_{j}-l_{k}}=0$ .

Case $(iii)$, if $l_{i}\cdot l_{k}=1$ and $l_{j}\cdot l_{k}=0$, then $(\ast) $
holds if $\mathfrak{\eta}_{l_{j},K^{\prime}-l_{i}-l_{k}}=0$.

Case $(iv)$, if $l_{i}\cdot l_{k}=1$ and $l_{j}\cdot l_{k}=1$, then $(\ast)$
holds if $\mathfrak{\eta}_{l_{i},K^{\prime}-l_{j}-l_{k}}\pm\mathfrak{\eta
}_{l_{j},K^{\prime}-l_{i}-l_{k}}=0$, here the sign is determined by the signs
of cubic form.

In conclusion, for any $l_{i},l_{j}\in I^{(E_{6},\mathbb{C}^{27})}$, if
$l_{i}\cdot l_{j}\neq0$, then $\eta_{i,j}=0$. If $l_{i}\cdot l_{j}=0$, then
$l_{i}-l_{j}=\alpha$ $(j>i)$ for $\alpha\in\Phi^{+}$, i.e. $\eta_{i,j}%
\in\Omega^{0,1}(Y,O(\alpha))$. And for any other $\eta_{p,q}\in\Omega
^{0,1}(Y,O(\alpha))$, we have $\eta_{i,j}\pm\eta_{p,q}=0$. From the signs of
the cubic form $c$, we know that given any positive root $\alpha$, there
exists $6$ $\eta_{i,j}$'s in $\Omega^{0,1}(Y,O(\alpha))$, where $3$ of them
are the same and the other $3$ different to the first three by a sign. We use
computer to prove we can find such $\eta_{i,j}$'s satisfying
$\mathfrak{\overline{\partial}}_{\mathfrak{\mathfrak{\eta}}}^{2}=0$.
\end{proof}
\end{proposition}

Until now, we have proved $\Xi_{Y}^{E_{6}}$ is not empty.

\begin{proposition}
\label{E6}The bundle $\mathfrak{L}_{\mathfrak{\eta}}^{E_{6}}$ over $Y$ with
$\mathfrak{\eta\in}\Xi_{Y}^{E_{6}}$ can descend to $X$ if and only if for
every $C_{k}$ and $\eta_{i,j}\in\Omega^{0,1}(Y,O(C_{k}))$, $[\eta
_{i,j}|_{C_{k}}]\neq0$, i.e. $\mathfrak{\eta\in}\Xi_{X}^{E_{6}}$.

\begin{proof}
From Lemma \ref{branch4}, Proposition \ref{c} and the order of $I^{(E_{6}%
,\mathbb{C}^{27})}$, for $\mathfrak{\eta\in}\Xi_{Y}^{E_{6}}$, $\mathfrak{L}%
_{\mathfrak{\eta}}^{E_{6}}$ can be constructed from $\mathfrak{L}%
_{\mathfrak{\eta}^{\prime}}^{A_{5}}$ for some $\mathfrak{\eta}^{\prime
}\mathfrak{\in}\Xi_{Y}^{A_{5}}$ with $\mathfrak{\eta}^{\prime}\subset
\mathfrak{\eta}$. Under the $($non-holomorphic$)$ direct sum decomposition
$\mathfrak{L}_{0}^{E_{6}}=\mathfrak{L}_{0}^{A_{5}}\oplus({\wedge}%
^{2}\mathfrak{L}_{0}^{A_{5}})^{\ast}(H)\oplus({\wedge}^{5}\mathfrak{L}%
_{0}^{A_{5}})^{\ast}(2H)$, $\overline{\partial}_{\mathfrak{\eta}}$ for
$\mathfrak{L}_{\mathfrak{\eta}}^{E_{6}}$ has the following block
decomposition:
\[
\left(
\begin{tabular}
[c]{c|c|c}%
$\overline{\partial}_{({\wedge}^{5}\mathfrak{L}_{\mathfrak{\eta}^{\prime}%
}^{A_{5}})^{\ast}(2H)}$ & $%
\begin{array}
[c]{cccc}%
\vdots & \vdots & \vdots & \ddots\\
\pm\mathfrak{\beta} & \ast & \ast & \cdots\\
0 & \pm\mathfrak{\beta} & \ast & \cdots\\
0 & 0 & \pm\mathfrak{\beta} & \cdots
\end{array}
$ & $\ast$\\\hline
$0$ & $\overline{\partial}_{({\wedge}^{2}\mathfrak{L}_{\mathfrak{\eta}%
^{\prime}}^{A_{5}})^{\ast}(H)}$ & $%
\begin{array}
[c]{cccc}%
\vdots & \vdots & \vdots & \ddots\\
\pm\mathfrak{\beta} & \ast & \ast & \cdots\\
0 & \pm\mathfrak{\beta} & \ast & \cdots\\
0 & 0 & \pm\mathfrak{\beta} & \cdots
\end{array}
$\\\hline
$0$ & $0$ & $\overline{\partial}_{\mathfrak{L}_{\mathfrak{\eta}^{\prime}%
}^{A_{5}}}$%
\end{tabular}
\ \ \ \ \ \right)  \text{.}%
\]
Here $\pm\mathfrak{\beta\in}\Omega^{0,1}(Y,O(C_{6}))$, it is because the
corresponding two $(-1)$-curves $l$ and $l^{\prime}$ satisfying $l-l^{\prime
}=C_{6}$. The signs of $\mathfrak{\beta}$ can be determined by $\overline
{\partial}_{\mathfrak{\eta}}c=0$.

From above, we know that $\mathfrak{L}_{\mathfrak{\eta}}^{E_{6}}|_{C_{k}}$
$(k\neq6)$ is trivial if and only if $\mathfrak{L}_{\mathfrak{\eta}^{\prime}%
}^{A_{5}}|_{C_{k}}$ $(k\neq6)$ is trivial. From Proposition \ref{A1}, we have
the theorem for $k\neq6$. For $C_{6}$, from Lemma \ref{indep4}, $\mathfrak{L}%
_{\mathfrak{\eta}}^{E_{6}}|_{C_{6}}$ is trivial if and only if these
$\pm\mathfrak{\beta}$'s satisfy $[\mathfrak{\beta}|_{C_{6}}]\neq0$.
\end{proof}
\end{proposition}

Note that $\mathcal{\zeta}_{\mathfrak{\eta}}^{E_{6}}=aut(\mathfrak{L}%
_{\mathfrak{\eta}}^{E_{6}},c)$ is an $E_{6}$ Lie algebra bundle over $Y$. In
order for $\mathcal{\zeta}_{\mathfrak{\eta}}^{E_{6}}$ to descend to $X$ as a
Lie algebra bundle, we need to show that $c|_{C_{i}}:\mathfrak{L}%
_{\mathfrak{\eta}}^{E_{6}}|_{C_{i}}\otimes\mathfrak{L}_{\mathfrak{\eta}%
}^{E_{6}}|_{C_{i}}\otimes\mathfrak{L}_{\mathfrak{\eta}}^{E_{6}}|_{C_{i}%
}\longrightarrow O_{C_{i}}\left(  K^{\prime}\right)  $ is a constant map for
every $C_{i}$. This follows from the fact that both $\mathfrak{L}%
_{\mathfrak{\eta}}^{E_{6}}$ and $O(K^{\prime})$ are trivial on all $C_{i}$'s
and $\overline{\partial}_{\mathfrak{\eta}}c=0$. From the construction,
$\mathfrak{L}_{\mathfrak{\eta}}^{E_{6}}$ is a representation bundle of
$\mathcal{\zeta}_{\mathfrak{\eta}}^{E_{6}}$.

The only other minuscule representation $\overline{{\mathbb{C}}^{27}}$ of
$E_{6}$ is the dual of the standard representation ${\mathbb{C}}^{27}$,
therefore $\mathfrak{L}_{\mathfrak{\eta}}^{(E_{6},\overline{\mathbb{C}^{27}}%
)}=\left(  \mathfrak{L}_{\mathfrak{\eta}}^{(E_{6},\mathbb{C}^{27})}\right)
^{\ast}$.

\subsection{$E_{7}$ case}

We recall that \cite{A} $E_{7}\mathbb{=}aut(\mathbb{C}^{56},t)$ for a
non-degenerate quartic form $t$ on the standard representation $\mathbb{C}%
^{56}$. There is no other minuscule representation of $E_{7}$.

We consider a surface $X$ with an $E_{7}$ singularity $p$ and a $(-1)$-curve
$C_{0}$ passing through $p$ with multiplicity $C_{1}$. By Proposition
\ref{card}, $I^{(E_{7},\mathbb{C}^{56})}$ has cardinality $56$. For any two
distinct $(-1)$-curves $l_{i}$ and $l_{j}$ in $I$, we have $l_{i}\cdot
l_{j}=0,1$ or $2$.

Define $\mathfrak{L}_{0}^{(E_{7},\mathbb{C}^{56})}:=\bigoplus_{l\in I}O(l)$
over $Y$, for simplicity, we write it as $\mathfrak{L}_{0}^{E_{7}}$. If we
ignore $C_{7}$, we recover the $A_{6}$ case as in $\S 4.1$.

\begin{lemma}
$\mathfrak{L}_{0}^{(E_{7},\mathbb{C}^{56})}=\mathfrak{L}_{0}^{A_{6}}%
\oplus({\wedge}^{2}\mathfrak{L}_{0}^{A_{6}})^{\ast}(H)\oplus({\wedge}%
^{5}\mathfrak{L}_{0}^{A_{6}})^{\ast}(2H)\oplus({\wedge}^{6}\mathfrak{L}%
_{0}^{A_{6}})^{\ast}(3H)$, where $H=3C_{0}^{1}+3C_{1}+3C_{2}+3C_{3}%
+3C_{4}+2C_{5}+C_{6}+C_{7}$.

\begin{proof}
\textit{Similar} to $E_{6}$ case.
\end{proof}
\end{lemma}

From the above lemma and direct computations, for any $C_{i}$,%
\[
\mathfrak{L}_{0}^{E_{7}}|_{C_{i}}\cong O_{\mathbb{P}^{1}}^{\oplus32}%
\oplus(O_{\mathbb{P}^{1}}(1)\oplus O_{\mathbb{P}^{1}}(-1))^{\oplus12}%
\]

The configuration of these $56$ $(-1)$-curves is as follows: Fix any
$(-1)$-curve, there are exactly $27$ $(-1)$-curves intersect it once, $1$
$(-1)$-curve intersects it twice. If $l_{i}+l_{j}+l_{p}+l_{q}=2K^{\prime}$
with $K^{\prime}=\ 2C_{0}^{1}+3C_{1}+4C_{2}+5C_{3}+6C_{4}+4C_{5}+2C_{6}%
+3C_{7}$, the four $(-1)$-curves $l_{i}$, $l_{j}$, $l_{p}$ and $l_{q}$ will
form a quadrangle.

From this configuration, we can define a quartic form $t$ on the vector space
$V_{0}=\mathbb{C}^{I}={\bigoplus_{l\in I}}\mathbb{C}\langle v_{l}\rangle$
spanned by all the $(-1)$-curves,%
\[
t:V_{0}\otimes V_{0}\otimes V_{0}\otimes V_{0}\longrightarrow\mathbb{C},\text{
}(v_{l_{i}},v_{l_{j}},v_{l_{p}},v_{l_{q}})\mapsto\left\{
\begin{tabular}
[c]{ll}%
$\pm1$ & if $l_{i}+l_{j}+l_{p}+l_{q}=2K^{\prime}$\\
$0$ & otherwise.
\end{tabular}
\ \ \ \ \right.
\]
The signs above can be determined explicitly \cite{A} such that $E_{7}%
=aut(V_{0},t)$.

Correspondingly, we have a fiberwise quartic form $t$ on the bundle
$\mathfrak{L}_{\mathfrak{\eta}}^{E_{7}}$,

\begin{center}
$t:\mathfrak{L}_{\mathfrak{\eta}}^{E_{7}}\otimes\mathfrak{L}_{\mathfrak{\eta}%
}^{E_{7}}\otimes\mathfrak{L}_{\mathfrak{\eta}}^{E_{7}}\otimes\mathfrak{L}%
_{\mathfrak{\eta}}^{E_{7}}\longrightarrow O\left(  2K^{\prime}\right)  $.
\end{center}

\begin{proposition}
\label{t}There exists $\mathfrak{\eta}$ with $\overline{\partial
}_{\mathfrak{\eta}}^{2}=0$ such that $\mathfrak{\overline{\partial}%
}_{\mathfrak{\mathfrak{\eta}}}t=0$.

\begin{proof}
Similar to $E_{6}$ case, but even more calculations involved. We will omit the
calculations here and only list the conditions for $\mathfrak{\overline
{\partial}}_{\mathfrak{\mathfrak{\eta}}}t=0$. From $\mathfrak{\overline
{\partial}}_{\mathfrak{\mathfrak{\eta}}}t=0$ we have when $l_{i}\cdot
l_{j}\neq0$, $\eta_{i,j}=0$. That means all the nonzero $\eta_{i,j}$'s are
corresponding to $l_{i}\cdot l_{j}=0$, then $l_{i}-l_{j}=\alpha$ for some root
$\alpha$, i.e. $\eta_{i,j}\in\Omega^{0,1}(Y,O(\alpha))$. Conversely, given any
positive root $\alpha$, there exists $12$ $\eta_{i,j}$'s in $\Omega
^{0,1}(Y,O(\alpha))$, where $6$ of them are the same and the other $6$
different to the first $6$ by a sign. We use computer to prove we can find
such $\eta_{i,j}$'s satisfying $\mathfrak{\overline{\partial}}%
_{\mathfrak{\mathfrak{\eta}}}^{2}=0$.
\end{proof}
\end{proposition}

Until now, we have proved $\Xi_{Y}^{E_{7}}$ is not empty.

\begin{proposition}
The bundle $\mathfrak{L}_{\mathfrak{\eta}}^{E_{7}}$ over $Y$ with
$\mathfrak{\eta\in}\Xi_{Y}^{E_{7}}$ can descend to $X$ if and only if for
every $C_{k}$ and $\eta_{i,j}\in\Omega^{0,1}(Y,O(C_{k}))$, $[\eta
_{i,j}|_{C_{k}}]\neq0$, i.e. $\mathfrak{\eta\in}\Xi_{X}^{E_{7}}$.

\begin{proof}
Similar to $E_{6}$ case $($Proposition \ref{E6}$)$.
\end{proof}
\end{proposition}

Note that $\mathcal{\zeta}_{\mathfrak{\eta}}^{E_{7}}=aut(\mathfrak{L}%
_{\mathfrak{\eta}}^{E_{7}},t)$ is an $E_{7}$ Lie algebra bundle over $Y$. In
order for $\mathcal{\zeta}_{\mathfrak{\eta}}^{E_{7}}$ to descend to $X$ as a
Lie algebra bundle, we need to show that $t|_{C_{i}}:\mathfrak{L}%
_{\mathfrak{\eta}}^{E_{7}}|_{C_{i}}\otimes\mathfrak{L}_{\mathfrak{\eta}%
}^{E_{7}}|_{C_{i}}\otimes\mathfrak{L}_{\mathfrak{\eta}}^{E_{7}}|_{C_{i}%
}\otimes\mathfrak{L}_{\mathfrak{\eta}}^{E_{7}}|_{C_{i}}\longrightarrow
O_{C_{i}}\left(  2K^{\prime}\right)  $ is a constant map for every $C_{i}$.
This follows from the fact that both $\mathfrak{L}_{\mathfrak{\eta}}^{E_{7}}$
and $O(2K^{\prime})$ are trivial on all $C_{i}$'s and $\overline{\partial
}_{\mathfrak{\eta}}t=0$. It is obvious that $\mathfrak{L}_{\mathfrak{\eta}%
}^{E_{7}}$ is a representation bundle of $\mathcal{\zeta}_{\mathfrak{\eta}%
}^{E_{7}}$.

\subsection{$E_{8}$ case}

Though $E_{8}$ has no minuscule representation, the fundamental representation
corresponding to $C_{1}$ is the adjoint representation of $E_{8}$.

We consider a surface $X$ with an $E_{8}$ singularity $p$ and a $(-1)$-curve
$C_{0}$ passing through $p$ with multiplicity $C_{1}$. By direct computations,
$|I|=240$. In this case, $l\in I$ if and only if $l-K^{\prime}\in\Phi$, where
$K^{\prime}=C_{0}^{1}+2C_{1}+3C_{2}+4C_{3}+5C_{4}+6C_{5}+4C_{6}+2C_{7}+3C_{8}%
$. So $\mathcal{E}_{0}^{E_{8}}$ defined in $\S 2$ can be written as follows:%
\[
\mathcal{E}_{0}^{E_{8}}:=O^{\oplus8}\oplus\bigoplus_{\alpha\in\Phi}%
O(\alpha)=(O(K^{\prime})^{\oplus8}\oplus\bigoplus_{l\in I}O(l))(-K^{\prime
})\text{.}%
\]

We will prove that $(\mathcal{E}_{\varphi}^{E_{8}},\overline{\partial
}_{\mathcal{\varphi}})$ with $\varphi=(\varphi_{\alpha})_{\alpha\in\Phi^{+}%
}\in\Psi_{X}$ descends to $X$ in $\S 7$.

\section{Proof of main results}

In the above three sections, we have constructed and studied the Lie algebra
bundles and minuscule representation bundles in $A_{n}$, $D_{n}$ and $E_{n}$
($n\neq8$) cases separately. We will prove the holomorphic structures on these
bundles can be expressed by forms in the positive root classes and the
representation actions.

\begin{proof}
$($\textit{of Theorem \ref{thm2} and \ref{thm3}}$)$\textit{ Recall when }%
$\rho:g\longrightarrow End(V)$\textit{ is the standard representation,
}$\mathfrak{L}_{\mathfrak{\eta}}^{(\mathfrak{g},V)}$\textit{ }$\mathit{(}%
\eta\in\Xi_{Y}^{\mathfrak{g}}\mathit{)}$\textit{ admits a holomorphic
fiberwise symmetric multi-linear form }$f$\textit{. And }$\overline{\partial
}_{\mathfrak{\eta}}f=0$\textit{ implies that }$\eta_{i,j}=0$\textit{ unless
}$l_{i}-l_{j}=\alpha$\textit{ }$\mathit{(}j>i\mathit{)}$\textit{ for some
}$\alpha\in\Phi^{+}$\textit{. Thus }$\eta_{i,j}=\varphi_{\alpha}\in
\Omega^{0,1}(Y,O(\alpha))$\textit{. Furthermore, if }$\eta_{i,j}\ $\textit{and
}$\eta_{i^{\prime},j^{\prime}}$\textit{ are in }$\Omega^{0,1}(Y,O(\alpha))$,
\textit{then they are the same up to sign. Thus we can write }$\eta
_{i,j}=n_{\alpha,w_{i}}\varphi_{\alpha}$\textit{, where }$n_{\alpha,w_{i}}%
$\textit{'s are as in }$\S 3$\textit{ since }$\rho$ \textit{preserves }%
$f$\textit{. Namely, }$\overline{\partial}_{\mathfrak{\mathfrak{\eta}}%
}=\overline{\partial}_{0}+\sum_{\alpha\in\Phi^{+}}c_{\alpha}\rho(x_{\alpha
})=\overline{\partial}_{0}+\sum_{\alpha\in\Phi^{+}}\rho(\varphi_{\alpha}%
)$\textit{ with }$\varphi_{\alpha}=c_{\alpha}x_{\alpha}$.

\textit{The holomorphic structure on the bundle }$\zeta_{\mathfrak{\eta}%
}^{\mathfrak{g}}:=aut(\mathfrak{L}_{\mathfrak{\eta}}^{(\mathfrak{g},V)}%
,f)$\textit{ is }$\overline{\partial}_{\mathcal{\mathfrak{\eta}}}%
=\overline{\partial}_{0}+\sum_{\alpha\in\Phi^{+}}c_{\alpha}ad(x_{\alpha}%
)$\textit{, which is the same as }$\overline{\partial}_{\mathcal{\varphi}}%
$\textit{ for }$\mathcal{E}_{\varphi}^{\mathfrak{g}}$\textit{ in }%
$\S 2$\textit{, i.e. }$\zeta_{\mathfrak{\eta}}^{\mathfrak{g}}=\mathcal{E}%
_{\varphi}^{\mathfrak{g}}$.

\textit{The only minuscule representations }$(\mathfrak{g},V)$\textit{ besides
standard representations are }$(A_{n},\wedge^{k}C^{n+1})$\textit{, }%
$(D_{n},S^{\pm})$\textit{ and }$(E_{6},$\textit{ }$\overline{{\mathbb{C}}%
^{27}})$\textit{. We denote corresponding actions as }$\rho$\textit{ as usual.
In each case, for }$\mathcal{E}_{\varphi}^{\mathfrak{g}}$\textit{ to act
holomorphically on the corresponding vector bundle, the holomorphic structure
on }$\mathfrak{L}_{0}^{(\mathfrak{g},V)}$\textit{ can only be }$\overline
{\partial}_{\mathfrak{\varphi}}$\textit{.}

\textit{The filtration of }$\mathfrak{L}_{0}^{(\mathfrak{g},V)}$\textit{ gives
one on }$\mathfrak{L}_{\mathfrak{\eta}}^{(\mathfrak{g},V)}$ \textit{since it
is constructed from extensions using elements in }$I_{i}\backslash I_{i+1}%
$\textit{ }$\mathit{(}\S 3.3\mathit{)}$\textit{.}

\textit{We note that all the above Lie algebra bundles and representation
bundles over }$Y$\textit{ can descend to }$X$\textit{ if and only if }%
$0\neq\lbrack\varphi_{C_{i}}|_{C_{i}}]\in H^{1}(Y,$\textit{ }$O_{C_{i}}%
(C_{i}))$\textit{ for all }$C_{i}$\textit{'s, i.e. }$\varphi=(\varphi_{\alpha
})_{\alpha\in\Phi^{+}}\in\Psi_{X}$\textit{.}
\end{proof}

\bigskip

From the above arguments, Theorem \textbf{\ref{thm1} }holds true for $ADE$
except $E_{8}$ case.

\begin{proof}
$($\textit{of Theorem \ref{thm1}}$)$\textit{ It remains to prove the }$E_{8}%
$\textit{ case. }%
\[
\mathcal{E}_{0}^{E_{8}}:=O^{\oplus8}\oplus\bigoplus_{\alpha\in\Phi}%
O(\alpha)=(O(K^{\prime})^{\oplus8}\oplus\bigoplus_{l\in I}O(l))(-K^{\prime
})\text{.}%
\]

\textit{We want to show that the bundle }$(\mathcal{E}_{\varphi}^{E_{8}%
},\overline{\partial}_{\mathcal{\varphi}})$\textit{ with }$\varphi
=(\varphi_{\alpha})_{\alpha\in\Phi^{+}}\in\Psi_{X}$\textit{ can descend to
}$X$\textit{, i.e. }$\mathcal{E}_{\varphi}^{E_{8}}|_{C_{i}}$\textit{ is
trivial for }$i=1,2,\cdots8$\textit{. Note }$O(K^{\prime})|_{C_{i}}$\textit{
is trivial for every }$i$\textit{, but }$O(l)|_{C_{i}}$\textit{ can be
}$O_{\mathbb{P}^{1}}(\pm2)$\textit{, hence Lemma \ref{section} is not
sufficient. However, if we ignore }$C_{8}$\textit{ }$\mathit{(}$\textit{resp.
}$C_{7}\mathit{)}$\textit{ in }$Y$\textit{, then we recover the }$A_{7}%
$\textit{ case }$\mathit{(}$\textit{resp. }$D_{7}$\textit{ case}$\mathit{)}%
$\textit{. Our approach is to reduce the problem of trivializing }%
$\mathcal{E}_{\varphi}^{E_{8}}|_{C_{i}}$\textit{ to one for a representation
bundle of }$A_{7}$\textit{ }$($\textit{resp. }$D_{7}\mathit{)}$\textit{.}

\textit{Step one, as }$A_{7}$\textit{ is a Lie subalgebra of }$E_{8}$\textit{,
the adjoint representation of }$E_{8}$\textit{ decomposes as a sum of
irreducible representations of }$A_{7}$\textit{. The branching rule is
}$248=8+28+56+64+56+28+8$\textit{, correspondingly, we have the following
decomposition of }$\mathcal{E}_{0}^{E_{8}}$\textit{ over }$Y$\textit{,}
\begin{align*}
\mathcal{E}_{0}^{E_{8}}  &  =\mathfrak{L}_{0}^{A_{7}}(-K^{\prime})\oplus
\wedge^{2}(\mathfrak{L}_{0}^{A_{7}})^{\ast}(H-K^{\prime})\oplus\wedge
^{5}(\mathfrak{L}_{0}^{A_{7}})^{\ast}(2H-K^{\prime})\oplus\\
&  \mathfrak{L}_{0}^{A_{7}}\otimes(\mathfrak{L}_{0}^{A_{7}})^{\ast}%
\oplus\wedge^{3}(\mathfrak{L}_{0}^{A_{7}})^{\ast}(H)\oplus\wedge
^{6}(\mathfrak{L}_{0}^{A_{7}})^{\ast}(2H)\oplus(\mathfrak{L}_{0}^{A_{7}%
})^{\ast}(K^{\prime})\text{,}%
\end{align*}
\textit{where }$H=3C_{0}^{1}+3C_{1}+3C_{2}+3C_{3}+3C_{4}+3C_{5}+2C_{6}%
+C_{7}+C_{8}$\textit{ and }$K^{\prime}=C_{0}^{1}+2C_{1}+3C_{2}+4C_{3}%
+5C_{4}+6C_{5}+4C_{6}+2C_{7}+3C_{8}$\textit{.}

\textit{Step two, instead of }$\mathfrak{L}_{0}^{A_{7}}$\textit{, we use
}$\mathfrak{L}_{\varphi}^{A_{7}}$\textit{ which is trivial on }$C_{i}$\textit{
for }$i\neq8$\textit{. We consider the bundle }%
\begin{align*}
\mathcal{E}^{^{\prime}E_{8}}  &  =\mathfrak{L}_{\varphi}^{A_{7}}(-K^{\prime
})\oplus\wedge^{2}(\mathfrak{L}_{\varphi}^{A_{7}})^{\ast}(H-K^{\prime}%
)\oplus\wedge^{5}(\mathfrak{L}_{\varphi}^{A_{7}})^{\ast}(2H-K^{\prime}%
)\oplus\\
&  \mathfrak{L}_{\varphi}^{A_{7}}\otimes(\mathfrak{L}_{\varphi}^{A_{7}}%
)^{\ast}\oplus\wedge^{3}(\mathfrak{L}_{\varphi}^{A_{7}})^{\ast}(H)\oplus
\wedge^{6}(\mathfrak{L}_{\varphi}^{A_{7}})^{\ast}(2H)\oplus(\mathfrak{L}%
_{\varphi}^{A_{7}})^{\ast}(K^{\prime})\text{.}%
\end{align*}
\textit{We have }$\overline{\partial}_{\mathcal{E}^{^{\prime}E_{8}}}%
=\overline{\partial}_{0}+\sum_{\alpha\in\Phi_{A_{7}}^{+}}ad(\varphi_{\alpha}%
)$\textit{. Since }$O(K^{\prime})$\textit{ and }$O(H)$\textit{ are both
trivial on }$C_{i}$\textit{ for }$i\neq8$, $\mathcal{E}^{^{\prime}E_{8}}%
$\textit{ is trivial on }$C_{i}$\textit{ for }$i\neq8$\textit{.}

\textit{Step three, we compare }$\mathcal{E}^{^{\prime}E_{8}}$\textit{ with
}$\mathcal{E}_{\varphi}^{E_{8}}$\textit{. Topologically they are the same.
Holomorphically, }%
\[
\overline{\partial}_{\mathcal{E}_{\varphi}^{E_{8}}}=\overline{\partial}%
_{0}+\sum_{\alpha\in\Phi_{E_{8}}^{+}}ad(\varphi_{\alpha})=\overline{\partial
}_{\mathcal{E}^{^{\prime}E_{8}}}+\sum_{\alpha\in\Phi_{E_{8}}^{+}\backslash
\Phi_{A_{7}}^{+}}ad(\varphi_{\alpha})\text{.}%
\]
\textit{If we write the holomorphic structure of }$\mathcal{E}_{\varphi
}^{E_{8}}$\textit{ as a }$248\times248$\textit{ matrix, then }$\varphi
_{\alpha}$\textit{ with }$\alpha\in\Phi_{E_{8}}^{+}\backslash\Phi_{A_{7}}^{+}%
$\textit{ must appear at those positions }$(\beta,\gamma)$\textit{ with
}$\beta-\gamma=\alpha$\textit{, where }$\beta$\textit{ has at least one more
}$C_{8}$\textit{ than }$\gamma$\textit{. That means, after taking extensions
between the summands of }$\mathcal{E}^{^{\prime}E_{8}}$\textit{, we can get
}$\mathcal{E}_{\varphi}^{E_{8}}$\textit{. Since }$\mathcal{E}^{^{\prime}E_{8}%
}$\textit{ is trivial on }$C_{i}$\textit{ for }$i\neq8$\textit{ and
}$Ext_{\mathbb{P}^{1}}^{1}(O,O)\cong0$\textit{ , we have }$\mathcal{E}%
_{\varphi}^{E_{8}}$\textit{ trivial on }$C_{i}$\textit{ for }$i\neq8$\textit{.
}

\textit{Similarly, if we consider the reduction of }$E_{8}$\textit{ to }%
$D_{7}$\textit{, from the branching rule }$248=14+64+1+91+64+14$\textit{, we
have the following decomposition of }$\mathcal{E}_{0}^{E_{8}}$\textit{, }%
\[
\mathcal{E}_{0}^{E_{8}}=\mathfrak{L}_{0}^{D_{7}}(-K^{\prime})\oplus
\mathfrak{L}_{0}^{(D_{7},{\mathcal{S}}^{+})}(C_{7}-C_{0}^{6})\oplus
O\oplus\mathcal{E}_{0}^{D_{7}}\oplus(\mathfrak{L}_{0}^{(D_{7},{\mathcal{S}%
}^{+})})^{\ast}(C_{0}^{6}-C_{7})\oplus(\mathfrak{L}_{0}^{D_{7}})^{\ast
}(K^{\prime})\text{.}%
\]
\textit{Instead of }$\mathfrak{L}_{0}^{D_{7}}$\textit{, we consider
}$\mathfrak{L}_{\varphi}^{D_{7}}$\textit{. Similar to the reduction to }%
$A_{7}$\textit{ case as above, we will get for }$(\mathcal{E}_{\varphi}%
^{E_{8}},\overline{\partial}_{\mathcal{\varphi}})$\textit{, if we take
}$[\varphi_{C_{i}}|_{C_{i}}]\neq0$\textit{, then }$\mathcal{E}_{\varphi
}^{E_{8}}$\textit{ is trivial on }$C_{i}$\textit{ for }$i\neq7$\textit{. Hence
we have proved Theorem \ref{thm1} for type }$E_{8}$\textit{.}
\end{proof}

\bigskip

\begin{proof}
$($\textit{of Theorem \ref{thm4}}$)$\textit{ We only need to find a divisor
}$B$\textit{ in }$Y$\textit{ such that }$\mathit{(i)}$\textit{ }$B$\textit{ is
a combination of }$C_{i}$\textit{'s and }$\widetilde{C_{0}}$\textit{ with the
coefficient of }$\widetilde{C_{0}}$\textit{ not zero, and }$\mathit{(ii)}%
$\textit{ }$O(B)$\textit{ can descend to }$X$\textit{. Then if we take }%
$k$\textit{ to be the coefficient of }$\widetilde{C_{0}}$\textit{ in }%
$B$\textit{, }$\mathbb{L}_{\mathfrak{\varphi}}^{(\mathfrak{g},V)}%
:=S^{k}\mathfrak{L}_{\mathfrak{\varphi}}^{(\mathfrak{g},V)}\otimes
O(-B)$\textit{ with }$\varphi\in\Psi_{X}$\textit{ can descend to }$X$\textit{
and does not depend on the existence of }$C_{0}$\textit{.}

$(A_{n},$\textit{ }$C^{n+1})$\textit{ case, }$B=(n+1)\widetilde{C_{0}}%
+nC_{1}+(n-1)C_{2}+\cdots+C_{n}$\textit{.}

$(A_{n},\wedge^{k}C^{n+1})$\textit{ case, }$B=(n+1)\widetilde{C_{0}%
}+(n-k+1)C_{1}+\cdots(k-1)(n-k-1)C_{k-1}+k(n-k)C_{k+1}+\cdots kC_{n}%
$\textit{.}

$(D_{n},$\textit{ }$C^{2n})$\textit{ case, }$B=F=2\widetilde{C_{0}}%
+2C_{1}+\cdots+2C_{n-2}+C_{n-1}+C_{n}$\textit{.}

$(D_{n},S^{+})$\textit{ case, }$B=4\widetilde{C_{0}}+2C_{1}+4C_{2}%
+\cdots+2(n-2)C_{n-2}+(n-2)C_{n-1}+nC_{n}$\textit{.}

$(E_{6},$\textit{ }$C^{27})$\textit{ case, }$B=3\widetilde{C_{0}}%
+4C_{1}+5C_{2}+6C_{3}+4C_{4}+2C_{5}+3C_{6}$\textit{.}

$(E_{7},$\textit{ }$C^{56})$\textit{ case, }$B=2\widetilde{C_{0}}%
+3C_{1}+4C_{2}+5C_{3}+6C_{4}+4C_{5}+2C_{6}+3C_{7}$\textit{.}
\end{proof}

\begin{remark}
We can determine Chern classes of the Lie algebra bundles and minuscule
representation bundles. For any minuscule representation bundle $\mathfrak{L}%
_{\mathfrak{\varphi}}^{(\mathfrak{g},V)}$,
\[
c_{1}(\mathfrak{L}_{\mathfrak{\varphi}}^{(\mathfrak{g},V)})=\sum_{l\in
I^{(\mathfrak{g},V)}}[l]\in H^{2}(Y,\mathbb{Z)}\text{.}%
\]

For any Lie algebra bundle $\mathcal{E}_{\varphi}^{\mathfrak{g}}$, we have%
\[
c_{1}(\mathcal{E}_{\varphi}^{\mathfrak{g}})=0
\]
and%
\[
c_{2}(\mathcal{E}_{\varphi}^{\mathfrak{g}})=\sum_{\alpha\neq\beta\in\Phi}%
c_{1}(O(\alpha))c_{1}(O(\beta))=\sum_{\alpha\in\Phi^{+}}c_{1}(O(\alpha
))c_{1}(O(-\alpha))=\dim(\mathfrak{g)-}rank(\mathfrak{g)}\text{.}%
\]
In particular, the bundles we defined above are not trivial.
\end{remark}

\begin{remark}
There are choices in the construction of our Lie algebra bundles and minuscule
representation bundles, we will see that these bundles are not unique. Take
$\mathfrak{L}_{\varphi}^{A_{2}}$ $(\varphi=(\varphi_{\alpha})_{\alpha\in
\Phi_{A_{2}}^{+}}\in\Psi_{X})$ as an example. The holomorphic structure on
$\mathfrak{L}_{\varphi}^{A_{2}}$ is as follows:%
\[
\overline{\partial}_{\varphi}=\left(
\begin{array}
[c]{ccc}%
\overline{\partial} & \varphi_{C_{2}} & \varphi_{C_{1}+C_{2}}\\
0 & \overline{\partial} & \varphi_{C_{1}}\\
0 & 0 & \overline{\partial}%
\end{array}
\right)
\]
with $[\varphi_{C_{1}}|_{C_{1}}]\neq0$ and $[\varphi_{C_{2}}|_{C_{2}}]\neq0$.
We replace $\varphi_{C_{1}+C_{2}}$ by $\varphi_{C_{1}+C_{2}}+\psi$, where
$\psi\in H^{1}(Y,O(C_{1}+C_{2}))\neq0$. If $[\psi]\neq0$, then $\overline
{\partial}_{\varphi+\psi}$ is not isomorphic to $\overline{\partial}_{\varphi
}$.
\end{remark}

\begin{remark}
Our $\mathfrak{g}$-bundle $\mathcal{E}_{\mathfrak{\eta}}^{\mathfrak{g}}$ over
$Y$ is given by $aut(\mathfrak{L}_{\mathfrak{\eta}}^{(\mathfrak{g},V)},f)$
with $f:\bigotimes^{r}\mathfrak{L}_{\mathfrak{\eta}}^{(\mathfrak{g}%
,V)}\longrightarrow O_{Y}(D)$. If $O(D)=O(rD^{\prime})$ for some divisor
$D^{\prime}$, then%
\[
f:\bigotimes^{r}\mathfrak{L}_{\mathfrak{\eta}}^{(\mathfrak{g},V)}(-D^{\prime
})\longrightarrow O_{Y}\text{.}%
\]
And $Aut(\mathfrak{L}_{\mathfrak{\eta}}^{(\mathfrak{g},V)}(-D^{\prime}),f)$ is
a Lie group bundle over $Y$ lifting $\mathcal{E}_{\mathfrak{\eta}%
}^{\mathfrak{g}}$. In general, we only have a $G\times\mathbb{Z}_{r}$-bundle,
or so-called conformal $G$-bundle in \cite{FM}.
\end{remark}

\section{Appendix}

We now construct examples of surface with an $ADE$ singularity $p$ of type
$\mathfrak{g}$ and a $(-1)$-curve $C_{0}$ passing through $p$ with minuscule
multiplicity $C_{k}$. We call its minimal resolution a surface with minuscule
configuration of type $(\mathfrak{g},V)$, where $V$ is the fundamental
representation corresponding to $-C_{k}$.

First we consider the standard representation $V\simeq\mathbb{C}^{n+1}$ of
$A_{n}=sl\left(  n+1\right)  $. When we blowup a point on any surface, the
exceptional curve is a $\left(  -1\right)  $-curve $E$. If we blowup a point
on $E$, the strict transform of $E$ becomes a $\left(  -2\right)  $-curve. By
repeating this process $n+1$ times, we obtain a chain of $\left(  -2\right)
$-curves with a $\left(  -1\right)  $-curve attached to the last one. Namely
we have a surface with a minuscule configuration of type $\left(
A_{n},\mathbb{C}^{n+1}\right)  $.

Suppose that $D$ is a smooth rational curve on a surface with $D^{2}=0$. By
blowing up a point on $D$, we obtain a surface with a chain of two $\left(
-1\right)  $-curves. If we blowup their intersection point and iterative
blowing up points in exceptional curves, then we obtain a surface with
minuscule configuration of type $\left(  D_{n},\mathbb{C}^{2n}\right)  $.

Given a surface together with a smooth rational curve $C$ with $C^{2}=1$ on
it. We could obtain every minuscule configuration by the following process. If
we blow up three points on $C$, then the strict transform of $C$ is an
$\left(  -2\right)  $-curve. By the previous construction of iterated blowups
of points in these three exceptional curves $E_{i}$'s, we could obtain many
minuscule configurations. Let us denote the number of iterated blowups of the
exceptional curve $E_{i}$ as $m_{i}$ with $i\in\left\{  1,2,3\right\}  $. Then
we can obtain minuscule configuration of type $(\mathfrak{g},V)$ by taking
suitable $m_{i}$'s as follows.%
\[%
\begin{tabular}
[c]{ll}%
minuscule configuration of type $(\mathfrak{g},V)$ & $\left(  m_{1}%
,m_{2},m_{3}\right)  $\\\hline
$\left(  A_{n},\Lambda^{k}\mathbb{C}^{n+1}\right)  $ for any $k$ & $\left(
k-1,0,n-k\right)  $\\
$\left(  D_{n},\mathbb{C}^{2n}\right)  $, $\left(  D_{n},S^{+}\right)  $ and
$\left(  D_{n},S^{-}\right)  $ & $\left(  n-3,1,1\right)  $\\
$\left(  E_{6},27\right)  $, $\left(  E_{6},\overline{27}\right)  $ & $\left(
2,1,2\right)  $\\
$\left(  E_{7},56\right)  $ & $\left(  3,1,2\right)  $%
\end{tabular}
\]

Note that we could obtain such a configuration for every adjoint
representation of $E_{n}$ this way. We remark that surfaces in this last
construction are necessarily rational surfaces because of the existence of $C
$ with $C^{2}=1$.

The Institute of Mathematical Sciences and Department of Mathematics, The
Chinese University of Hong Kong, Shatin, N.T., Hong Kong

E-mail address: yxchen@math.cuhk.edu.hk

E-mail address: leung@math.cuhk.edu.hk
\end{document}